\documentclass[11pt,twoside]{article}

\usepackage{amsmath,amsfonts,latexsym}

\usepackage{times}
\usepackage{graphicx}
\usepackage{amsthm}
\usepackage[colorlinks,pdfdisplaydoctitle]{hyperref}
\newtheorem{theorem}{Theorem}
\newtheorem{lemma}[theorem]{Lemma}
\newtheorem{assumption}[theorem]{Assumption}
\newtheorem{example}[theorem]{Example}
\newtheorem{remark}[theorem]{Remark}

\DeclareMathOperator{\diag}{diag}
\DeclareMathOperator{\Proj}{P}
\newcommand{\argmin}{\operatornamewithlimits{argmin}}

\newcommand{\B}{\mathrm B}
\newcommand{\R}{\mathrm R}
\newcommand{\mc}[1]{\mathcal #1}
\newcommand{\bs}[1]{\boldsymbol #1}

\newcommand{\bbN}{\mathbb N}
\newcommand{\bbR}{\mathbb R}

\newcommand{\bbP}{\mathbb P}
\newcommand{\bbJ}{\mathbb J}

\newcommand{\bbE}{\mathbb E}

\newcommand{\Cov}{\mathrm{Cov}}
\newcommand{\ev}[1]{\bbE \left[  #1 \right]}

\newcommand{\prob}{\bbP}

\newcommand{\E}{\mathrm e}
\newcommand{\D}{\mathrm d}

\begin{document}
\title{Analysis of the Ensemble and\\ Polynomial Chaos Kalman Filters in\\ Bayesian Inverse Problems}
\author{Oliver G. Ernst$^\dag$, Bj\"orn Sprungk$^\dag$ and Hans-J\"org Starkloff$^\ast$\\
\small
$^\dag$ Department of Mathematics, TU Chemnitz, Germany\\
\small
$^\ast$ Fachgruppe Mathematik, University of Applied Sciences Zwickau, Germany
}
\date{}
\maketitle

%
%
%
%

%
%

\begin{abstract}
We analyze the Ensemble and Polynomial Chaos Kalman filters applied to nonlinear stationary Bayesian inverse problems.
In a sequential data assimilation setting such stationary problems arise in each step of either filter.
We give a new interpretation of the approximations produced by these two popular filters in the Bayesian context and prove that,
in the limit of large ensemble or high  polynomial degree, both methods yield approximations which converge to a well-defined random variable termed the analysis random variable.
We then show that this analysis variable is more closely related to a specific linear Bayes estimator than to the solution of the associated Bayesian inverse problem given by the posterior measure.
This suggests limited or at least guarded use of these generalized Kalman filter methods for the purpose of uncertainty quantification.
\end{abstract}

%

\pagestyle{myheadings}
\thispagestyle{plain}
\markboth{O. G. Ernst, B. Sprungk, H.-J. Starkloff}{Kalman Filters and Bayesian Inference}

%

%
%
\section{Introduction}

Due to increasing attention to uncertainty quantification (UQ) for complex systems, in particular as relates to the study and solution of partial differential equations (PDEs) with random data, interest has also focussed on inverse problems for random PDEs.
In particular, the Bayesian approach to inverse problems has become popular in this context.
From a UQ perspective the inverse problem is of tremendous interest since incorporating any available information
into the probability law of an uncertain quantity will, in general, reduce uncertainty and lead to improved stochastic models.

We consider in this work the fundamental task of inferring knowledge about an unknown element $u \in \mc X$ from a separable Hilbert space $\mc X$ by observing finite-dimensional noisy data
\begin{equation} \label{equ:ip}
	z = G(u) + \varepsilon,
\end{equation}
where $G: \mc X \to \bbR^d$ denotes the known (and deterministic) parameter-to-solution map and $\varepsilon$ the observational noise.
Adopting the Bayesian perspective, we assume a probability measure $\mu_0$ on $\mc X$ to be given describing our \emph{prior} knowledge or belief about $u$ which may be based, e.g., on physical reasoning, expert knowledge or previously collected data.
We wish to highlight the distinction between the two main tasks associated with Bayesian inverse problems, namely \emph{identification} and \emph{inference}, where the latter may include the former.
Identification refers to the task of determining an element $\hat u \in \mc X$ which best explains the observed data $z$ in accordance with given a priori assumptions, yielding a best guess or best single approximation to the unknown $u$.
By inference we mean the gain in knowledge by merging a prior probabilistic model $\mu_0$ with new information $z \in \bbR^d$ to obtain an updated  model $\mu^z$ which represents the new understanding or belief about $u$.

This incorporation of new information is realized mathematically by conditioning the prior probability measure $\mu_0$ on the event $\{G(u)+\varepsilon = z\}$ and is thus rooted in Kolmogorov's fundamental concept of \emph{conditional expectation} \cite{Rao2010}.
\emph{Bayes' rule} provides an analytic expression for the resulting conditioned or posterior distribution in terms of the prior distribution and provides the main tool in Bayesian inference and Bayesian inverse problems (BIPs).

While BIPs enjoy a number of favorable theoretical properties compared with their deterministic counterparts, i.e., they are well-posed and their solution in the form of the a posteriori measure is, in a certain sense, explicitly characterized, they do pose significant computational challenges in that they entail calculations with highly correlated and complex distributions in high-dimensional spaces.
The primary ``workhorse'' here is the Markov Chain Monte Carlo (MCMC) method \cite{Geyer2011}, whose continued improvement drives a very active field of research. 
However, MCMC simulations can be quite costly, since the chain has to run long enough to give sufficiently accurate estimates and each iteration typically requires one evaluation of the forward map $G$, e.g., one PDE solve.
Thus, for online monitoring or control of complex dynamical systems such as arise in weather forecasting or oil reservoir management, MCMC methods are prohibitively expensive, and filtering methods like the Kalman filter or the Ensemble Kalman filter are often applied to the associated state or parameter estimation problem.
Moreover, in dynamical systems where observational data arrives sequentially in time, Kalman filter-type methods provide the significant advantage that their recursive structure is adapted to this sequential availability of data (see \cite[Section 5.4]{Stuart2010} for a nice discussion of this issue).
So far, the Kalman filter (KF) \cite{Kalman1960} and its generalizations have mainly been used for state estimation, i.e., for identification rather than for Bayesian inference for quantifying uncertainty.
In recent years, however, these methods have drawn the attention of the growing UQ community, e.g. \cite{IglesiasEtAl2013a, IglesiasEtAl2013b, LawStuart2012}, and are being increasingly applied also to Bayesian inverse problems.
The point of departure is typically the Ensemble Kalman filter (EnKF) \cite{Evensen1994}, an extension of the KF to nonlinear models of type \eqref{equ:ip}.
As an example of this development, the authors of \cite{BlanchardEtAl2010, PajonkEtAl2012, RosicEtAl2012, RosicEtAl2013, SaadEtAl2007, SaadGhanem2009} have combined the idea of the EnKF with the computationally attractive representation of random variables in a polynomial chaos expansion to develop an efficient method for Bayesian inverse problems.
In place of (deterministic) state estimation, these methods model the uncertain state as a random variable which is updated with the arrival of each new set of observations.
We will refer to this approach in the following as the Polynomial Chaos Kalman filter (PCKF).
It was the study of this new PCKF method which motivated this work, because, although its authors gave a motivation for deriving their algorithm, the random variable approximated by the PCKF is not clearly characterized.
The same is true for the EnKF: Despite its many documented applications a detailed description of the nature or distribution of the analysis ensemble produced by one EnKF update is still lacking.
Only occasional hints that the EnKF generally fails to yield an ensemble distributed according to the posterior measure can be found in the literature.

The present work fills this gap and clarifies the stochastic model underlying the EnKF and PCKF.
We determine the precise quantities approximated by the EnKF and PCKF and how these approximations relate to the solution of Bayesian inverse problems and Bayes estimators.
In addition, we prove convergence results for both methods in the limit of increasing ``resolution'', i.e., for large ensemble size for the EnKF and large polynomial degree for the PCKF, respectively.
The question of convergence of the EnKF or PCKF is also missing in the literature so far.
To the authors' knowledge, the only related result is \cite{MandelEtAl2011}, where the convergence of the EnKF applied to data assimilation in linear, dynamical systems was studied.

The remainder of this paper is organized as follows:
Section~\ref{sec:BIP} briefly recalls the Bayesian approach to inverse problems as well as Bayes estimators.
In Section \ref{sec:GKF_Analysis} we describe and analyze the EnKF and PCKF.
In particular, we prove that the approximations provided by these generalized Kalman filtering methods converge to a certain \emph{analysis random variable}.
A characterization of this analysis random variable in light of Bayes estimators is further given in Section \ref{sec:GKF_Bayes} where we show that its distribution, in general, differs from the desired posterior measure.
Moreover, we illustrate the performance of the EnKF and PCKF and the difference between their approximations and the solution of the Bayesian inverse problem for a simple 1D boundary value problem and a simple dynamical system in Section \ref{sec:Exams}.
Section \ref{sec:Conclusions} provides a summary and conclusion.

%
%
\section{Bayesian Inverse Problems and Bayes Estimators} \label{sec:BIP}

In this section we introduce the basic concepts of the Bayesian approach to inverse problems.
Throughout, let $|\cdot|$ denote the Euclidean norm on $\bbR^d$, $\|\cdot\|$ the norm and $\langle \cdot, \cdot \rangle$ the inner product in a general separable Hilbert space $\mc X$, 
and $\mc Y$ a second separable Hilbert space.
By $\mc L(\mc X, \mc Y)$ we denote the set of all bounded linear operators $A:\mc X \to \mc Y$.
Note that $\mc L(\mc X, \mc Y)$ is isometrically isomorphic to the tensor product of the Hilbert spaces $\mc X \otimes \mc Y$ \cite{LightCheney1985, ReedSimon1980}.

In order to regularize the usually ill-posed least-squares formulation
\[
	u = \argmin_{v\in\mc X} |z - G(v)|^2
\]
of the inverse problem \eqref{equ:ip}, one incorporates additional prior information about the desired $u$ into the (deterministic) identification problem by way of a regularization functional \cite{EnglEtAl2000}, $\R:\mc X \to [0,\infty]$, and solves for
\[
	u_\alpha = \argmin_{v\in\mc X} |z - G(v)|^2 + \alpha \, \R(v),
\]
where $\alpha \in [0,\infty)$ serves as a regularization parameter to be chosen wisely \cite{AnzengruberEtAl2013}.
A further possibility for regularization is to restrict $u$ to a subset or subspace $\tilde{\mc X} \subset \mc X$, e.g., by using a stronger norm of $u$ as the regularization functional.
Broadly speaking, the Bayesian approach may be viewed as yet another way of modelling prior information on $u$ and adding it to the inverse problem.
In this case we express our prior belief about $u$ through a probability distribution $\mu_0$ on the Hilbert space $\mc X$, by which a quantitative preference of some solutions $u$ over others may be given by assigning higher and lower probabilities.
However, the goal in the Bayesian approach is not the \emph{identification} of one specific $u\in\mc X$, but rather \emph{inference} on $u$, i.e., we would like to \emph{learn from the data} in a statistical or probabilistic
sense by \emph{adjusting our prior belief} $\mu_0$ about $u$ in accordance with the newly available data $z$.
The task of identification may also be achieved within the Bayesian framework through \emph{Bayes estimates} and \emph{Bayes estimators}, which are discussed in Section~\ref{sec:bayes_estimators}.

In the Bayesian setting the deterministic model \eqref{equ:ip} becomes
\begin{equation} \label{equ:ip_bayes}
	Z = G(U) + \varepsilon,
\end{equation}
where now $\varepsilon$, and hence $Z$, are $\bbR^d$-valued random variables.
For the unknown random variable $U$ with values in $\mc X$ and prior probability distribution $\mu_0$, we seek the posterior probability distribution given the available observations $Z=z$.
Before giving a precise definition of the posterior distribution we require some basic concepts from probability theory.

%
\subsection{Probability Measures and Random Variables}

Let $(\Omega, \mc F, \bbP)$ denote a probability space and $\mc B(\mc X)$ the Borel $\sigma$-algebra of $\mc X$ generated by the open sets in $\mc X$ w.r.t. $\|\cdot\|$.
A measurable mapping $X: (\Omega, \mc F) \to (\mc X, \mc B(\mc X))$ is called a random variable (RV) and the measure $\bbP_X := \bbP \circ X^{-1}$, i.e., $\bbP_X(A) = \bbP(X^{-1}(A))$ for all $A\in\mc B(\mc X)$, defines the distribution of $X$ as the push-forward measure of $\mathbb P$ under $X$.
Conversely, given a probability measure $\mu$ on $(\mc X, \mc B(\mc X))$, we mean by $X \sim \mu$ that $\bbP_X = \mu$.
Further, let $\sigma(X) \subset \mc F$ denote the $\sigma$-algebra generated by $X$, i.e., $\sigma(X) = \{X^{-1}(A): A \in \mc B(\mc X)\}$.

The Bochner space of $p$-integrable $\mc X$-valued RVs, i.e., the space of (equivalence classes of) RVs $X: \Omega \to \mc X$ such that $\int_{\Omega} \|X(\omega)\|^p \, \bbP(\D \omega) < \infty$, is denoted by $L^p(\Omega, \mc F, \bbP; \mc X)$ or simply $L^p(\mc X)$ when the context is clear.

An element $m\in \mc X$ is called the \emph{mean} of a RV $X$ if for any $x \in \mc X$ there holds $\langle x, m \rangle = \bbE[\langle x, X\rangle ]$.
Here and in the following $\bbE$ denotes the expectation operator w.r.t.\ $\bbP$.
If $X \in L^1(\Omega, \mc F, \bbP; \mc X)$ then its mean is given by the Bochner integral $m = \bbE[X] = \int_{\Omega} X(\omega) \, \bbP(\D \omega)$.
A bilinear form $C: \mc X \times \mc Y \to \bbR $ is called the \emph{covariance} $\Cov(X,Y)$ of two RVs $X: \Omega \to \mc X$ and $Y:\Omega \to \mc Y$ if it satisfies $C(x,y) = \bbE\big[ \langle x, X - \bbE[X]\rangle \, \langle y, Y - \bbE[Y]\rangle \big]$ for all $x \in \mc X$ and $y\in\mc Y$, and we set $\Cov(X) := \Cov(X,X)$.
We shall also employ the identity $\Cov(X,Y) = \bbE[ (X-\bbE[X]) \otimes (Y-\bbE[Y])]$ when convenient.
The covariance $\Cov(X,Y)$ can also be defined equivalently as an operator $\hat C: \mc X \to \mc Y$ such that $\langle \hat C x, y\rangle = C(x,y)$.
We will mainly work with the latter definition in the following but on occasion will also apply the tensor product form $\bbE[ (X-\bbE[X]) \otimes (Y-\bbE[Y])]$.
The definitions of mean and covariance extend to RVs with values in separable Banach spaces by considering the topological duals of $X$ and $Y$, respectively.

We also require the notion of distance between probability measures, one of which is given by the \emph{Hellinger metric} $d_H$: given two probability measures $\mu_1$ and $\mu_2$ on the Hilbert space $\mc X$, it is defined as
\[
	d_H(\mu_1, \mu_2)
	:=
	\left[ \int_{\mc X}
	       \left( \sqrt{ \frac{\D \mu_1}{\D \nu}(u) }
	         - \sqrt{ \frac{\D \mu_2}{\D \nu}(u) } \right)^2
	       \, \nu(\D u)
    \right]^{1/2},
\]
where $\nu$ is a dominating measure of $\mu_1$ and $\mu_2$, e.g., $\nu = (\mu_1+\mu_2)/2$.
Note that the definition of the Hellinger metric is independent of the dominating measure.
Another metric for probability measures which we will employ in the following is the \emph{Wasserstein metric}
\[
	d_W(\mu_1, \mu_2)
	:=
	\sup_{\mathrm{Lip}(f) \leq 1}
	\left| \int_{\mc X} f(u) \, \mu_1 (\D u)
	- \int_{\mc X} f(u) \, \mu_2 (\D u) \right|,
\]
where the supremum is taken over all $f:\mc X \to \bbR$ which satisfy $|f(u) - f(v)| \leq \|u-v\|$.
For relations of the Hellinger and Wasserstein metrics to other probability metrics such as total variation distance, we refer to \cite{GibbsSu2002}. \par

In the following, we will use upper case latin letters such as $X$, $Y$, $Z$, $U$ to denote RVs on Hilbert spaces and lower case latin letters like $x$, $y$, $z$, $u$ for elements in these Hilbert spaces or realizations of the associated RVs, respectively.
Greek letters such as $\varepsilon$, $\eta$ and $\xi$ will be used to denote RVs as well as their realizations and $\mu$ and $\nu$ (with various subscripts) will denote measures on the Hilbert space $\mc X$ and $\bbR^d$, respectively.

%
\subsection{Bayes' Rule and the Posterior Measure}

Bayesian inference consists in updating
our prior knowledge on the unknown quantity $U$,
reflecting a gain in knowledge due to new observations.
%
The distribution of the RV $U$, characterized by the probabilities $\bbP(U \in B)$ for $B \in \mc B(\mc X)$, quantifies in stochastic terms our knowledge about the uncertainty associated with $U$.
When new information becomes available, such as knowing that the event $Z=z$ has occurred, this is reflected in our quantitative description as the ``conditional distribution of $U$ given $\{Z = z\}$'', denoted $\bbP(U\in B | Z=z)$.
Unfortunately, $\bbP(U\in B | Z=z)$ cannot be defined in an elementary fashion when $\bbP(Z=z)=0$, in which case the  conditional distribution is defined by an integral relation.
The key concept here is that of \emph{conditional expectation}:
Given RVs $X \in L^1(\Omega,\mc F,\bbP;\mc X)$ and $Y: \Omega \to \mc Y$, we define the conditional expectation $\bbE[X | Y]$ of $X$ given $Y$ as any $\sigma(Y)$-measurable mapping $\bbE[X | Y]:\Omega \to \mc X$ which satisfies
\[
	\int_A \bbE[X | Y] \; \bbP(\D \omega) = \int_A X \; \bbP(\D \omega) \qquad \forall A \in \sigma(Y).
\]
By the Doob-Dynkin Lemma \cite[Lemma 1.13]{Kallenberg1997} there exists a measurable function $\phi: \mc Y \to \mc X$ such that $\bbE[X|Y] = \phi(Y)$ $\bbP$-almost surely.
We note that this does not determine a unique function $\phi$ but rather an equivalence class of measurable functions, where $\phi_1 \sim \phi_2$ iff $\bbP(Y\in \{ y \in \mc Y : \phi_1(y) \neq\phi_2(y) \}) = 0$.
For a specific realization $y$ of $Y$ (and a specific $\phi$), we
define
\[
   \bbE[X|Y=y] := \phi(y)\in \mc X.
\]
Setting $X = \mathbf{1}_{\{U\in B\}}$, one can then define for each fixed $B \in \mc B(\mc X)$
\begin{equation} \label{equ:cond-dist}
	\bbP(U \in B | Z=z) := \bbE[ \mathbf{1}_{\{U\in B\}} | Z=z]
\end{equation}
as an equivalence class of measurable functions $\bbR^d \to [0,1]$.
One would like to view this, conversely, as a family of probability measures with the realization $z$ as a parameter, giving the posterior distribution of $U$ resulting from having made the observation $Z=z$.
Unfortunately, this construction need not, in general, yield a probability measure for each fixed value of $z$ (cf. \cite{Rao2010}).
In case $\mc X$ is a separable Hilbert space, a function
\[
   Q : \mc B(\mc X) \times \bbR^d \to \mathbb R
\]
can be shown to exist (cf. \cite{Rao2010}) such that
\begin{enumerate} \renewcommand{\theenumi}{(\alph{enumi})}
\item \label{rcm1}
For each $z \in \bbR^d$, $Q(\cdot,z)$ is a probability measure on $(\mc X,\mc B(\mc X))$.
\item \label{rcm2}
For each $B \in \mc B(\mc X)$ the function
\[
   \bbR^d \ni z \mapsto Q(B,z)
\]
is a representative of the equivalence class
\eqref{equ:cond-dist}, i.e., it is measurable and there holds
\[
   \bbP(U \in B, Z \in A) = \int_A Q(B,z) \; \bbP_Z(\D z)
   \qquad \forall A \in \mc B(\bbR^d).
\]
\end{enumerate}
Such a function $Q$, also denoted by $\mu_{U|Z}$, is called the \emph{regular conditional distribution} of $U$ given $Z$
and is defined uniquely up to sets of $z$-values of $\bbP_Z$-measure zero.
We have thus arrived at a consistent definition of the posterior probability
$\bbP(U \in B | Z=z)$ as $\mu_{U|Z}(B,z)$.\\

It is helpful to maintain a clear distinction between \emph{conditional} and \emph{posterior} quantities: the former contain the -- as yet unrealized -- observation as a parameter, while in the latter the observation has been made.
Specifically, $\mu_{U|Z}$ is the conditional measure of $U$ conditioned on $Z$, whereas $\mu_{U|Z}(\cdot,z)$ denotes the posterior measure of $U$ for the observation $Z=z$.\\

We now recall how Bayes' rule yields an explicit expression for the regular conditional distribution $\mu_{U|Z}$.
To this end, we make the following assumptions for the model \eqref{equ:ip_bayes}.

\begin{assumption} \label{assum:A1}\hfill
\begin{enumerate}
\item
$U \sim \mu_0$, $\varepsilon \sim \nu_\varepsilon$ and $(U,\varepsilon) \sim \mu_0 \otimes \nu_\varepsilon$, i.e., $U$ and $\varepsilon$ are independent.

\item
$\nu_\varepsilon = \rho(\varepsilon) \, \D\varepsilon$ where $\rho(\varepsilon) = C \E^{- \ell(\varepsilon) }$ with $C>0$ and $\ell:\bbR^d \to \bbR^+_0$ measurable and nonnegative.
Here $\D\varepsilon$ denotes Lebesgue measure on $\bbR^d$.

\item
$G: \mc X \to \bbR^d$ is continuous.
\end{enumerate}
\end{assumption}
By Assumption~\ref{assum:A1}, the distribution $\nu_Z$ of $Z$ in \eqref{equ:ip_bayes} is determined as $\nu_Z =C\gamma(z)\D z$ where $C>0$ and
\[
	\gamma(z) := \int_{\mc X} \E^{-\ell(z- G(u))} \, \mu_0(\D u).
\]
We note that $\gamma(z)>0$ is well-defined since $0<|\E^{-\ell(z- G(u))}| \leq 1$ and $\gamma \in L^1(\bbR^d)$ due to Fubini's theorem \cite[Theorem 1.27]{Kallenberg1997}.
In particular, we have that $(U,Z) \sim \mu$ with $\mu(\D u, \D z) = C\E^{-\ell(z- G(u))} \, \mu_0(\D u) \otimes \D z$ where $\D z$ again denotes Lebesgue measure on $\bbR^d$.
Further, we introduce the \emph{potential}
\[
	\Phi(u; z) := \ell(z - G(u)),
\]
for which we assume the following Lipschitz-like property:

\begin{assumption} \label{assum:A2} 
The potential $\Phi$ is continuous in $z$ in mean-square sense w.r.t.\ $\mu_0$, i.e, there exists a nondecreasing function $\psi:[0,\infty) \to [0,\infty)$ with $\lim_{s \to 0} \psi(s) = \psi(0) = 0$ such that
\[
	\bbE \left[ | \Phi(U; z_1) - \Phi(U; z_2) |^2 \right] = \int_\mathcal{X} | \Phi(u; z_1) - \Phi(u; z_2) |^2 \, \mu_0(\D u) \leq \psi( |z_1-z_2| ).
\]
For instance, there may exist a function $\theta \in L^2(\mc X, \mc B(\mc X), \mu_0; \bbR)$ such that
\[
	| \Phi(u; z_1) - \Phi(u; z_2) | \leq \theta(u)  \, \psi( |z_1-z_2| ).
\]
\end{assumption}
Before stating the abstract version of Bayes' Rule in
Theorem~\ref{theo:cont_hellinger},
we recall the finite-dimensional case $\mc X \simeq \bbR^n$ where it can be stated in terms of densities: here $\mu_0(\D u) = \pi_0(u) \D u$, and \emph{Bayes' rule} takes the form
\[
	\pi^z(u) = \frac{1}{\gamma(z)}\,\exp(-\Phi(u; z)) \,  \pi_0(u)
\]
where $\E^{-\Phi(u; z)} = \E^{-\ell(z - G(u))}$ represents the \emph{likelihood} of observing $z$ when fixing $u$.
The denominator $\gamma(z)$ can be interpreted as a \emph{normalizing constant} to ensure $\int_{\mc X} \pi^z(u) \, \D u = 1$.
We now show that, in the general setting, Bayes' rule yields (a version of) the (regular) conditional measure $\mu_{U|Z}$ of $U$ w.r.t.\ $Z$.
The statement of Theorem \ref{theo:cont_hellinger} differs from related results in \cite[Theorem 4.2 and 6.31]{Stuart2010} insofar as we explicitly characterize the posterior measure as a version of the regular conditional distribution and as we allow also for a general prior $\mu_0$ and log-likelihood $\ell$.

\begin{theorem} \label{theo:cont_hellinger}
Let Assumptions \ref{assum:A1} and \ref{assum:A2} be satisfied and define for each $z \in \bbR^d$ a probability measure on $(\mc X, \mc B(\mc X))$ by
\begin{equation}\label{equ:posterior}
	\mu^z(\D u) := \frac 1{\gamma(z)} \exp(- \Phi(u; z))\; \mu_0(\D u).
\end{equation}
Then the mapping $Q : \mc B(\mc X) \times \bbR^d \to [0,1]$ given by
\[
	Q(B, z) := \mu^z(B) \qquad \forall B\in \mc B(\mc X)
\]
is a regular conditional distribution of $U$ given $Z$.
We call $\mu^z$ the \emph{posterior measure (of $U$ given $Z=z$)}. 
Moreover,
$\mu^z$ depends continuously on $z$ w.r.t. the Hellinger metric, i.e., for any $z_1,z_2\in\bbR^d$ with $|z_1 - z_2| \leq r$ there holds
\[
	d_H(\mu^{z_1}, \mu^{z_2}) \leq C_r(z_1)  \, \psi( |z_1-z_2| ),
\]
where $C_r(z_1) = C (1 + \min\{\gamma(z'): |z_1- z'|\leq r\}^3 )^{-1}< +\infty$.
\end{theorem}
\begin{proof}
Continuity with respect to the Hellinger metric is a slight generalization of \cite[Theorem 4.2]{Stuart2010} and may be proved in the same way with obvious modifications.
To show that $Q$ is a regular conditional distribution we verify the two properties \ref{rcm1} and \ref{rcm2}. 
The first follows from the construction of $\mu^z$.
For the second property, note that measurability follows from continuity.
The continuity of $\mu^z$ w.r.t. $z$ in the Hellinger metric implies also that $\mu^z(B)$ depends continuously on $z$ due to the relations between Hellinger metric and total variation distance (see \cite{GibbsSu2002}).
Finally, we have for any $A \in \mc B(\bbR^d)$ and $B\in \mc B(\mc X)$ that
\begin{align*}
   \bbP(U \in B, Z \in A)
   & =
   \int_{A\times B} \mu(\D u, \D z)
   =
   \int_{A} \int_B C\E^{-\ell(z- G(u))} \, \mu_0(\D u) \, \D z \\
   & =
   \int_A C \gamma(z) Q(B,z) \, \D z
   =
   \int_A Q(B,z) \; \bbP_Z(\D z)
\end{align*}
which completes the proof. \hfill
\end{proof}

\begin{remark}
Theorem~\ref{theo:cont_hellinger} shows that the Lipschitz-like property of the potential stated in Assumption \ref{assum:A2} carries over to the posterior for a general prior $\mu_0$ and an additive error  $\varepsilon$ with Lebesgue density proportional to $\E^{-\ell(\varepsilon)}$.
Roughly speaking, the negative log-likelihood $\ell$ and the posterior $\mu^z$ share the same local modulus of continuity.
\end{remark}

By Theorem \ref{theo:cont_hellinger} the Bayesian inverse problem is well-posed under mild conditions.
It is also possible to prove continuity of $\mu^z$ w.r.t.\ to the forward map $G$, see \cite[Section 4.4]{Stuart2010}, which is crucial when the forward map $G$ is realized by numerical approximation.

To give meaning to the mean and covariance of $U\sim \mu_0$ and $Z = G(U) + \varepsilon$, we make the further assumption that all second moments exist:
\begin{assumption} \label{assum:A3}
There holds
\[
	\int_{\mc X} \left(\right.\|u\|^2 + |G(u)|^2 \left.\right) \, \mu_0(\D u) < +\infty
	\quad \text{ and } \quad
	\int_{\bbR^d} |\varepsilon|^2 \, \nu_{\varepsilon}(\D \varepsilon) < +\infty.
\]
\end{assumption}

%
\subsection{Bayes Estimators} \label{sec:bayes_estimators}

Although the posterior measure $\mu^z$ is, by definition, the solution to the Bayesian inverse problem, it is by no means easy to compute in practice.
In special cases, such as when $G$ is linear and $\mu_0$ and $\nu_\varepsilon$ are Gaussian measures, or in the case of conjugate priors, closed-form expressions for $\mu^z$ are available.
In general, however, $\mu^z$ can only be computed in an approximate sense.
Moreover, when the dimension of $\mc X$ is large or infinite, visualizing, exploring or using $\mu^z$ for post-processing are demanding tasks.

More accessible quantities from Bayesian statistics \cite{Bernardo2003}, which are also closer in nature to the result of deterministic parameter identification procedures than the posterior measure, are \emph{point estimates} for the unknown $u$.
In the Bayesian setting a point estimate is a ``best guess'' $\hat u$ of $u$ based on posterior knowledge.
Here ``best'' is determined by a \emph{cost function} $c: \mc X \to \bbR_0^+$ satisfying $c(0) = 0$ and $c(u) \leq c(\lambda u)$ for any $u\in \mc X$ and  $\lambda \geq 1$.
This cost function describes the loss or costs $c(u - \hat u)$ incurred when $\hat u$ is substituted for (the true) $u$ for post-processing or decision making.
Also more general forms of a cost function are possible, see, e.g.,  \cite{Berger1985,Bernardo2003}.

For any realization $z \in \bbR^d$ of the observation RV $Z$ we introduce the \emph{(posterior) Bayes cost} of the estimate $\hat u$ w.r.t.\ $c$ as
\[
	\B_c(\hat u; z) := \int_{\mc X} c( u - \hat u)\, \mu^z(\D u),
\]
and define the \emph{Bayes estimate} $\hat u$ as a minimizer of this cost, i.e.,
\[
		\hat u := \argmin_{v \in \mc X} \B_c(v; z),
\]
assuming a unique minimizer exists.
The \emph{Bayes estimator} $\hat \phi: \bbR^d \to \mc X$ is then the mapping which assigns to an observation $z$ the associated Bayes estimate $\hat u$, i.e.,
\[
	\hat \phi : z \mapsto \argmin_{v \in \mc X} \B_c(v; z).
\]
We assume measurability of $\hat \phi$ in the following and
note that $\hat \phi$ is then also the minimizer of the \emph{expected} or \emph{prior Bayes cost}
\[
	\B_c(\phi)
	:=
	\ev{\B_c(\phi(Z); Z)}
	= \int_{\bbR^d} \int_{\mc X} c( u - \phi(z) )\, \mu^z(\D u) \, \nu_Z(\D z)
	= \ev{c(U - \phi(Z))},
\]
i.e., for any other measurable $\phi:\bbR^d \to \mc X$ there holds 
\[
	\ev{c(U - \hat \phi(Z))} \leq \ev{c(U - \phi(Z))}.
\]

\begin{remark}
Since $\hat \phi = \argmin_{\phi} \ev{c(U - \phi(Z))}$ it is possible to determine the estimator $\hat \phi$, and hence also the estimate $\hat u = \hat \phi(z)$ for a given $z$, without actually computing the posterior measure $\mu^z$, as the integration in $\B_c(\hat \phi)$ is carried out w.r.t.\ the prior measure.
Therefore, Bayes estimators are typically easier to compute or approximate than $\mu^z$.
\end{remark}

We now introduce two very common Bayes estimators: the \emph{posterior mean estimator} and the \emph{maximum a posteriori estimator}.

%
%
\subsubsection{Posterior Mean Estimator}

For the cost function $c(u) = \|u\|^2$ the posterior Bayes cost
\[
	\B_c(\hat u; z) = \int_{\mc X} \| u - \hat u\|^2 \, \mu^z(\D u)
\]
is minimized by the posterior mean $\hat u = u_\mathrm{CM} := \int_{\mc X} u \, \mu^z(\D u)$, since for any Hilbert space-valued RV $X$ its expectation $\bbE[X]$ is the minimizer of the functional $J_X(v) = \bbE[\|X - v\|^2]$, $v\in\mc X$.
The corresponding Bayes estimator for $c(u) = \|u\|^2$ is then given by
\[
	\hat \phi_\mathrm{CM}(z) := \int_{\mc X} u \, \mu^z(\D u).
\]
In particular, $\hat \phi_\mathrm{CM}(Z) = \bbE[U | Z ]$ holds $\bbP$-almost surely. 

\begin{remark} \label{rem:EX_in_BS}
If $\mc X$ is only a Banach space then the expectation of an $\mc X$-valued RV $X$ need not minimize the functional $J_X$, i.e., we have in general
\[
	\bbE[X] \neq \argmin_{v\in\mc X} \bbE[\|X - v\|^2].
\]
As a simple counterexample, consider $\mc X = \bbR^2$, $\|v\| = |v_1| + |v_2|$ and $X = (X_1,X_2)$ with independent random variables $X_1$, $X_2$ such that
\[
    \bbP(X_1 = -1) = p_1, \; \bbP(X_1 = 1) = 1-p_1 \text{ and }
    \bbP(X_2 = -1) = p_2, \; \bbP(X_2 = 1) = 1-p_2.
\]
Here $\bbE[X]$ minimizes $\bbE[\|X - v\|^2]$ iff $p_1 = p_2 = 0.5$.
In fact, one can show $\bbE[X] = \argmin_{v\in\mc X} \bbE[\|X - v\|^2]$ if $X$ is distributed symmetrically w.r.t its mean, i.e., if there holds $\bbP(X-\bbE[X] \in A) = \bbP(\bbE[X]-X \in A)$ for all $A\in\mc B(\mc X)$.

\end{remark}

%
%
\subsubsection{Maximum A Posteriori Estimator}

Another common estimator in Bayesian statistics is the \emph{maximum a posteriori (MAP)} estimator $\hat \phi_\mathrm{MAP}$.
For finite-dimensional $\mc X \simeq \bbR^n$ and absolutely continuous prior $\mu_0$, i.e., $\mu_0(\D u) = \pi_0(u) \D u$, the MAP estimate is defined as
\[
	\hat \phi_\mathrm{MAP}(z) = \argmin_{u \in \bbR^n} \Phi(u;z) - \log \pi_0(u)
\]
provided the minimum exists for all $z\in\bbR^d$.
For the definition of the MAP estimate via a cost function and the Bayes cost, we refer to the literature, e.g., \cite[Section 16.2]{LewisEtAl2006} or the very recent work \cite{BurgerLucka2014} for a novel approach;
for MAP estimates in infinite dimensions, we refer to \cite{DashtiEtAl2013}.

There is an interesting link between the Bayes estimator $\hat \phi_\mathrm{MAP}$ and the solution of the associated regularized least-squares problem:
If $\R: \bbR^n \to [0,\infty)$ is a
regularizing functional which satisfies $\int_{\bbR^n}  \exp( - \frac {\alpha}{ \sigma^ 2} \, \R(u))\, \D u < + \infty$, then the solution $\hat u_\alpha = \argmin_{u} |z - G(u)|^2 + \alpha \R(u)$ coincides with the MAP estimate $\hat \phi_\mathrm{MAP}(z)$ for $\varepsilon \sim N(0, \sigma^2 I)$ and $\mu_0(\D u) \propto \exp( - \frac {\alpha}{ \sigma^ 2} \, \R(u))\; \D u$.

%
%
\section{Analysis of Generalized Kalman Filters} \label{sec:GKF_Analysis}

In this section we consider Kalman filters and their application to the nonlinear Bayesian inverse problem \eqref{equ:ip_bayes}.
We begin with the classical Kalman filter for state estimation in linear dynamics and then consider two generalizations to the nonlinear setting which have been recently proposed for UQ in connection with inverse problems.
We show that both methods can be understood as different discretizations of an updating scheme for a certain RV and prove that both Kalman filter methods converge to this RV when the discretization is refined.

%
\subsection{The Kalman Filter} \label{sec:KF}

The Kalman filter \cite{Kalman1960} is a well-known method for sequential state estimation for incompletely observable, linear discrete-time dynamics
\begin{eqnarray}  \label{equ:kalman_system}
\begin{split}
	U_{n} = A_nU_{n-1} + \eta_n, \qquad
	Z_{n}  = G_{n} U_{n} + \varepsilon_{n}, \qquad n=1, 2, \ldots,
\end{split}
\end{eqnarray}
where $(U_n)_{n\in\bbN}$ denotes the unknown, unobservable state
and $(Z_n)_{n\in\bbN}$ the observable process.
The operators $A_n$ and $G_n$ are linear mappings in state space and from state to observation space, respectively, and the noise
processes  $\eta_n$, $\varepsilon_n$ are usually assumed
to have zero mean with known covariances.
In addition, the mean and covariance of $U_0$ need to be known and the RVs $U_0$, $\eta_n$, $\varepsilon_n$ are taken to be mutually independent.
Then, given observations $Z_1 = z_1, \ldots, Z_n=z_n$, the Kalman filter yields recursive equations for the minimum variance estimates $\hat u_n$ of $U_n$ and their error covariances $\Cov(U_n - \hat u_n)$, see, e.g., \cite{Catlin1989, Simon2006} for an  introduction and discussion.

Although the main advantage of the Kalman filter is its recursive structure, making it very efficient for state estimation in dynamical systems with sequentially arriving data, a detailed analysis of sequential methods is beyond the scope of this work.
We focus instead on the application of the Kalman filter and its generalizations to time-independent systems of the form \eqref{equ:ip_bayes} and, in the linear case,
\begin{equation} \label{equ:ip_kalman}
	Z = GU + \varepsilon, \qquad (U, \varepsilon) \sim \mu_0 \otimes \nu_\varepsilon.
\end{equation}
We note that \eqref{equ:ip_kalman} can be seen as one step of the dynamical system \eqref{equ:kalman_system} for $A_n \equiv I$, $\eta_n \equiv 0$ and $G_n = G$.
Conversely, the state estimation problem for $U = U_0$, $U = U_n$ or $U=(U_0,U_1,\ldots,U_n)$ in \eqref{equ:kalman_system} given $Z = (Z_1, \ldots, Z_n) = (z_1, \ldots, z_n) = z$ can be reformulated as \eqref{equ:ip_kalman}.

If $\hat u_0 = \bbE[U]$ is taken as an initial estimate for the unkown $U$ in \eqref{equ:ip_kalman} before observing $Z = z$, this results in the initial error covariance $\Cov(U - \hat u_0) = \Cov(U) =: C_0$.
Given data $Z=z$, the Kalman filter provides a new estimate $\hat u_1$ and its error covariance $C_1 = \Cov(U - \hat u_1)$ via the updates
\begin{align} \label{equ:kalman_filter}
	\hat u_1 & = \hat u_0 + K(z - G\hat u_0), \qquad
	C_1 = C_0 - KGC_0,
\end{align}
where $K = C_0G^*(GC_0G^* + \Sigma )^{-1}$, $\Sigma = \Cov(\varepsilon)$, is known as the \emph{Kalman gain}.
In fact, by assimilating the data $Z=z$ the Kalman filter produces an improved estimate, since its expected error is smaller than that of the initial estimate in the sense that $C_0 - C_1$ is positive definite.
\hfill \par
If $(U,Z)$ are jointly Gaussian RV, i.e., $U \sim N(m_0, C_0)$ and $\varepsilon \sim N(0, \Sigma)$, the posterior measure $\mu^z$ of $U$ given $Z=z$ also has a Gaussian distribution $\mu^z \sim N(m^z, C^z)$ with
\[
	m^z = m_0 + K(z - Gm_0), \qquad C^z = C_0 - KGC_0,
\]
see, e.g., \cite{Mandelbaum1984}.
Thus for $G$ linear and $U, \varepsilon$ independently Gaussian, the Kalman filter is seen to yield the solution of the Bayesian inverse problem by providing the posterior mean and covariance, which in this case also uniquely specify the Gaussian posterior measure $\mu^z$.
However, we emphasize that the Kalman filter does not directly approximate the posterior measure, it rather provides minimum variance estimates and their error covariances for linear problems \eqref{equ:kalman_system}.
Without the assumption that $\mu_0$ or $\nu_\varepsilon$ are Gaussian the Kalman filter will not, in general, yield the first two posterior moments, nor is the posterior measure necessarily Gaussian.\\ \par
\noindent In the following two subsections we consider generalizations of the Kalman filter to nonlinear problems \eqref{equ:ip_bayes}.
The historically first such method was the extended Kalman filter (EKF), which is based on local linearizations of the nonlinear map $G$, but which we will not consider here.
We rather focus on the Ensemble Kalman Filter (EnKF) introduced by Evensen \cite{Evensen1994} and the recently developed the Polynomial Chaos Kalman Filter (PCKF). 

%
\subsection{The Ensemble Kalman Filter}

Since its introduction in 1994, the EnKF has been investigated and evaluated in many publications \cite{Evensen2003, BurgersEtAl1998, Evensen2009a, Evensen2009b, MyrsethOmre2011}.
However, the focus is usually on its application to state or parameter estimation rather than solving Bayesian inverse problems.
Recently, the interest in the EnKF for UQ in inverse problems has increased, see, e.g., \cite{IglesiasEtAl2013b, IglesiasEtAl2013a, LawStuart2012}.\par

If we consider the model $Z = G(U) + \varepsilon$ with $(U, \varepsilon) \sim \mu_0 \otimes \nu_\varepsilon$ and given observations $z\in\bbR^d $, the EnKF algorithm proceeds as follows:

\begin{enumerate}
\item
\textbf{Initial ensemble:} Draw samples $u_1, \ldots, u_M$ of $U \sim \mu_0$.

\item
\textbf{Forecast:} Draw samples $\varepsilon_1, \ldots, \varepsilon_M$ of $\varepsilon \sim \nu_\varepsilon$, set
	\[
		z_j = G(u_j) + \varepsilon_j, \qquad  j = 1,\ldots,M,
	\]
yielding samples $z_1, \ldots, z_M$ of $Z\sim\nu_Z$.
\item
\textbf{Analysis:} Update the inital ensemble $\boldsymbol{u} = (u_1, \ldots, u_M)$ member by member via
	\begin{equation} \label{equ:EnKF_update}
		u^a_j = u_j + \tilde K(z -  z_j), \qquad  j = 1,\ldots,M,
	\end{equation}
	where $\tilde K = \Cov(\boldsymbol{u}, \boldsymbol{z}) \Cov(\boldsymbol{z})^{-1}$ and $\Cov(\boldsymbol{u}, \boldsymbol{z})$ and $\Cov(\boldsymbol{z})$ are the empirical covariances of the samples $\boldsymbol{u}$ and $\boldsymbol{z}=(z_1,\ldots,z_M)$, e.g.,
\[
	\Cov(\boldsymbol{u}, \boldsymbol{z}) = \frac 1{M-1} \sum_{j=1}^M (u_j - \bar{\bs u}) \otimes (z_j - \bar{\bs z}),
\]
	where $\bar{\bs u} = \frac 1M (u_1+\cdots+u_M)$ and $\bar{\bs z} = \frac 1M (z_1+\cdots+z_M)$.
	This yields an \emph{analysis ensemble} $\boldsymbol{u}^a = (u^a_1, \ldots, u^a_M)$ which in turn determines an \emph{empirical analysis measure}
	\begin{equation} \label{equ:EnKF_measure}
		\tilde \mu_M^a = \frac 1M \sum_{j=1}^M \delta_{u^a_j},
	\end{equation}
	where $\delta_{u^a_j}$ denotes the Dirac-measure at the point $u^a_j$.
	Moreover, the empirical mean of $\boldsymbol{u}^a$ serves as an estimate $\hat u$ for the unknown $u$ and the empirical covariance of $\boldsymbol{u}^a$ as an indicator for the accuracy of the estimate.
\end{enumerate}
For dynamical systems such as \eqref{equ:kalman_system}, the analysis ensemble $\boldsymbol{u}^a$ would be propagated by the system dynamics and would then serve as the initial ensemble for the subsequent step $n$.

%
\subsection{The Polynomial Chaos Kalman Filter}

In \cite{BlanchardEtAl2010, PajonkEtAl2012, RosicEtAl2012, RosicEtAl2013, SaadEtAl2007, SaadGhanem2009}
the authors propose a sampling-free Kalman filtering scheme for nonlinear systems.
Rather than updating samples of the unknown, this is carried out for the coefficient vector of a polynomial chaos expansion (PCE) of the unknown.
This necessitates the construction of a PCE distributed according to the prior measure $\mu_0$:
we assume there exist countably many independent real-valued random variables $\boldsymbol{\xi} = (\xi_m)_{m\in \bbN}$, and \emph{chaos coefficients} $u_{\boldsymbol{\alpha}} \in \mc X$, $\varepsilon_{\boldsymbol \alpha} \in \bbR^d$ for each multi-index
\[
	\boldsymbol{\alpha} \in \bbJ
	:=
	\{\boldsymbol{\alpha} \in \bbN_0^\bbN: \alpha_j \neq 0 \text{ for only finitely many } j\},
\]
such that
\[
	\sum_{\boldsymbol{\alpha} \in \bbJ} \|u_{\boldsymbol{\alpha}}\|^2 < +\infty
	\quad \text{ and } \quad
	\sum_{\boldsymbol{\alpha} \in \bbJ} |\varepsilon_{\boldsymbol{\alpha}}|^2 < +\infty,	 
\]
and
\[
	\Big( \sum_{\boldsymbol{\alpha} \in \bbJ} u_{\boldsymbol{\alpha}} P_{\boldsymbol{\alpha}} (\boldsymbol{\xi}), \quad
	\sum_{\boldsymbol{\alpha} \in \bbJ} \varepsilon_{\boldsymbol{\alpha}} P_{\boldsymbol{\alpha}}(\boldsymbol{\xi}) \Big)
	\sim \mu_0 \otimes \nu_\varepsilon.
\]
Here, $P_{\boldsymbol{\alpha}}(\boldsymbol{\xi}) = \prod_{m \geq 1} P^{(m)}_{\alpha_m}(\xi_m)$ denotes the product of univariate orthogonal polynomials $P^{(m)}_{\alpha_m}$ where we require $\{P^{(m)}_\alpha\}_{\alpha\in\bbN}$ to be a CONS in $L^2(\bbR, \mc B(\bbR) , \bbP_{\xi_m}; \bbR)$.
We note that the completeness of orthogonal polynomials will depend in general on properties of the measure $\bbP_{\xi_m}$, see \cite{ErnstEtAl2012} for a complete characterization.

We then define
$U := \sum_{\boldsymbol{\alpha} \in \bbJ} u_{\boldsymbol{\alpha}} P_{\boldsymbol{\alpha}}(\boldsymbol{\xi})$
and
$\varepsilon := \sum_{\boldsymbol{\alpha} \in \bbJ} \varepsilon_{\boldsymbol{\alpha}} P_{\boldsymbol{\alpha}}(\boldsymbol{\xi})$,
given the chaos coefficients
$(u_{\boldsymbol{\alpha}})_{\boldsymbol{\alpha} \in \bbJ}$
and
$(\varepsilon_{\boldsymbol{\alpha}})_{\boldsymbol{\alpha} \in \bbJ}$.
However, for numerical simulations we have to truncate the PCE and, therefore, introduce
the projection
\[
	\Proj_J U := \sum_{\boldsymbol{\alpha} \in J} u_{\boldsymbol{\alpha}} P_{\boldsymbol{\alpha}}(\boldsymbol{\xi}), \qquad  J \subset \bbJ.
\]
To simplify notation we further define for $J \subseteq \bbJ$ the following two RVs
\[
	U_J := \Proj_J U \quad \text{ and } \quad Z_J := \Proj_J (G(U_J) + \varepsilon).
\]
Due to the nonlinearity of $G$ there holds in general $\Proj_J G(U) \neq G(\Proj_J U) \neq \Proj_J G(U_J)$, and, hence, $Z_J \neq \Proj_J Z$!
In particular, we will consider finite subsets $J$, and for convergence studies we usually assume a monotone and exhaustive sequence of such finite subsets $(J_n)_{n\in\bbN}$, i.e, $J_m \subset J_n$ for $m \leq n$ and $J_n \uparrow \bbJ$, e.g.,
\[
	J_n 
	:= 
	\biggl\{\boldsymbol \alpha \in \bbJ: \alpha_j = 0\; \forall j > n, \sum_{j=1}^\infty |\alpha_j| \leq n\biggr\}.
\]
We note that for $n\to\infty$ the error $\|U - U_{J_n}\|_{L^2(\mc X)}$ will tend to zero since $J_n \uparrow \bbJ$.
However, the $L^2$-convergence is in general not preserved under continuous mappings (unlike convergence in the almost sure sense, in probability and in distribution).
Thus, although there holds $\|U - U_{J_n}\|_{L^2(\mc X)}\to 0$ and, of course, $\|G(U) - \Proj_{J_n} G(U)\|_{L^2(\bbR^d)} \to 0$, the continuity of $G$ does not imply $\|G(U) - \Proj_{J_n} G(U_{J_n})\|^2_{L^2(\bbR^d)} \to 0$ in general.
However, if we assume for a $\delta > 0$ that there exists $C < +\infty$ such that
\begin{align} \label{equ:G_bound_assum}
	\bbE \left[ |G(U_{J_n})|^{2+\delta} \right] \leq C \qquad \forall n\in\bbN,
\end{align}
the desired convergence of $\|Z - Z_{J_n}\|_{L^2(\bbR^d)} \to 0$ follows, see the proof of Theorem \ref{theo:PCEKF_conv} for details.\par

For the same problem considered for the EnKF, the PCKF algorithm now reads as follows:

\begin{enumerate}
\item
\textbf{Initialization:} Choose a finite subset $J\subset \bbJ$ and compute the chaos coefficients $(u_{\boldsymbol{\alpha}})_{\boldsymbol{\alpha} \in J}$ of $U\sim \mu_0$.

\item
\textbf{Forecast:} Compute the chaos coefficients $(g_{J, \boldsymbol{\alpha}})_{\boldsymbol{\alpha} \in J}$ of $G(U_J)$ and set
\[
	z_{J, \boldsymbol{\alpha}} 
	:= 
	g_{J, \boldsymbol{\alpha}} + \varepsilon_{\boldsymbol{\alpha}} 
	\qquad \forall \boldsymbol{\alpha} \in J,
\]
where $(\varepsilon_{\boldsymbol{\alpha}})_{\boldsymbol{\alpha} \in J}$ are the chaos coefficients of $\varepsilon$.
By linearity $z_{J, \boldsymbol{\alpha}}$ are the chaos coefficients of $Z_J$.
\item
\textbf{Analysis:} Update the inital chaos coefficients by
\begin{equation} \label{equ:PCEKF_update}
  u^a_{J, \boldsymbol{\alpha}} 
  := 
  u_{\boldsymbol{\alpha}} + K_J\left( \delta_{\boldsymbol{\alpha}\boldsymbol{0}}z - z_{J,\boldsymbol{\alpha}} \right) 
  \qquad \forall \boldsymbol{\alpha} \in J,
\end{equation}
where $\delta_{\boldsymbol{\alpha}\boldsymbol{0}}$ is the Kronecker symbol for multi-indices, $(\delta_{\boldsymbol{\alpha}\boldsymbol{0}}z)_{\boldsymbol{\alpha} \in J} = (z, 0, \ldots, 0)$ the chaos coefficients of the observed data $z\in\bbR^d$ and $K_J:  = \Cov(U_J, Z_J)\Cov(Z_J)^{-1}$.
The action of the covariances as linear operators can be described in the case of $\Cov(U_J, Z_J):\bbR^d \to \mc X$ by
\[
	\Cov(U_J, Z_J)x
	=
	\sum_{\boldsymbol{\alpha}\in J} \sum_{{\boldsymbol{\beta}}\in J}
	    z^\top_{J,{\boldsymbol{\beta}}} \, x \, u_{\boldsymbol{\alpha}}, \qquad x\in\bbR^d.
\]
\end{enumerate}
Thus, the result of one step of the PCKF algorithm is an \emph{analysis chaos coefficient vector} $(u^a_{\boldsymbol{\alpha}})_{\boldsymbol{\alpha}\in J}$, which in turn determines a RV
\[
	U^a_J :=  \sum_{\boldsymbol{\alpha} \in J} u^a_{J, \boldsymbol{\alpha}} P_{\boldsymbol{\alpha}}(\boldsymbol{\xi}).
\]

\begin{remark}
An expansion in polynomials $P_{\boldsymbol{\alpha}} (\boldsymbol{\xi})$ is not crucial for the application of the PCKF.
In principle, any countable CONS $(\Psi_\alpha)_{\alpha\in\bbN}$ of the space $L^2(\bbR^\bbN, \mc B(\bbR^\bbN) , \bbP_{\boldsymbol{\xi}};\bbR)$ such that $\big(\sum_{\alpha} u_{\alpha} \Psi_{\alpha} (\boldsymbol{\xi}), \sum_{\alpha} \varepsilon_{\alpha} \Psi_{\alpha} (\boldsymbol{\xi}) \big) \sim \mu_0 \otimes \nu_\varepsilon$ would be suitable.
\end{remark}

\subsection{The Analysis Variable}

Both EnKF and PCKF perform discretized versions of an update for RVs, namely,
\begin{align} \label{equ:analysis_variable}
	U^a = U + K(z - Z), \qquad K = \Cov(U,Z)\Cov(Z)^{-1},
\end{align}
where $Z:= G(U) + \varepsilon$ and $(U, \varepsilon) \sim \mu_0 \otimes \nu_\varepsilon$, providing samples $\boldsymbol{u}^a$ or chaos coefficients $u^a_{\boldsymbol{\alpha}}$ of $U^a$, respectively.
However, the output of both methods is corrupted by the approximation of the Kalman gain operator $K$ by the empirical covariances and the operator $K_J$, respectively.
That both methods do indeed converge to $U^a$ in some sense for increasing sample size $M$ or increasing chaos coefficient subset $J_n$ is shown by the next two theorems.

\begin{theorem} \label{theo:PCEKF_conv} 
Consider the model \eqref{equ:ip_bayes} and let Assumptions  \ref{assum:A1}, \ref{assum:A2} and \ref{assum:A3} be satisfied.
If $(J_n)_{n\in\bbN}$ is a monotone and exhaustive sequence of finite subsets of $\bbJ$ with $\boldsymbol 0 \in J_1$ such that \eqref{equ:G_bound_assum} holds, then $\|Z-Z_{J_n}\|_{L^2(\bbR^d)} \to 0$ for $n\to \infty$. 
Moreover, if
\[
	U^a_{J_n} 
	= 
	\sum_{\boldsymbol{\alpha} \in J_n} 
	u^a_{J_n, \boldsymbol{\alpha}} P_{\boldsymbol{\alpha}}(\boldsymbol{\xi}),
\]
denotes the RV generated by the PCKF in the analysis step for the subset $J=J_n$, 
we have
\begin{equation}
	\|U^a - U^a_{J_n}\|_{L^2(\mc X)}  \in \mc O \left( \|U-U_{J_n}\|_{L^2(\mc X)} + \|Z-Z_{J_n}\|_{L^2(\bbR^d)} \right),
\end{equation}
which means in particular that $U^a_{J_n} \to U^a$ in $L^2(\mc X)$ as $n\to\infty$.
\end{theorem}
\vspace{1em}
\begin{proof}
In the following we use $\|\cdot\|_{L^2}$ as shorthand for $\|\cdot\|_{L^2(\mc X)}$ and $\|\cdot\|_{L^2(\bbR^d)}$, respectively.
Since $(J_n)_{n\in\bbN}$ is exhaustive, we have $U_{J_n} \to U$ in $L^2(\mc X)$, and hence $U_{J_n} \xrightarrow{\prob} U$, where $\xrightarrow{\prob}$ denotes convergence in probability.
Since $G$ is continuous, it follows by the continuous mapping theorem \cite[Lemma 3.3]{Kallenberg1997} that also $G(U_{J_n}) \xrightarrow{\prob} G(U)$.
Now the boundedness assumption \eqref{equ:G_bound_assum} implies the uniform integrability of the RVs $|G(U_{J_n})|^2$, $n\in\bbN$, see  \cite[p. 44]{Kallenberg1997}, and by \cite[Proposition 3.12]{Kallenberg1997} we then obtain $G(U_{J_n}) \to G(U)$ in $L^2(\mc X)$.
Thus,
\[
	\|Z - Z_{J_n}\|_{L^2}
	\leq
	\underbrace{\|Z - \Proj_{J_n}Z\|_{L^2}}_{\to 0}
	+ \underbrace{\|\Proj_{J_n}(Z - G(U_{J_n}) - \varepsilon)\|_{L^2}}_{\leq \|G(U) - G(U_{J_n})\|_{L^2} \to 0}
	\to 0.
\]
Now consider $J$ as an arbitrary subset.
Since $U^a = U + K(z - Z)$ and $U_J^a = U_J + K_J(z - Z_J)$, we have
\begin{eqnarray*}
	\|U^a - U^a_J\|_{L^2} &  \leq & \|U - U_J\|_{L^2} + \|K - K_J\| \; \|z-Z_J\|_{L^2} + \|K\| \; \|Z - Z_J\|_{L^2},
\end{eqnarray*}
where the norm for $K$ and $K-K_J$ is the usual operator norm for linear mappings from $\bbR^d \to \mc X$.
It is clear that we can estimate
\[
	\|z-Z_J\|_{L^2} \leq |z| + \|Z\|_{L^2},
\]
because $\|Z_J\|_{L^2} \leq \|Z\|_{L^2}$.
Considering $\|K - K_J\|$, we can further split this error into
\begin{eqnarray*}
	\|K - K_J\| & \leq & \|\Cov(U,Z) - \Cov(U_J,Z_J)\| \; \|\Cov^{-1}(Z)\| \\
	& & \quad + \|\Cov(U_J,Z_J)\| \; \|\Cov^{-1}(Z) - \Cov^{-1}(Z_J)\|.
\end{eqnarray*}
Next, we recall that the covariance operator $\Cov(X,Y)$ depends continuously on $X$ and $Y$, in particular we have for zero-mean Hilbert space-valued RV $X_1,X_2 \in L^2(\mc X)$ and $Y_1, Y_2 \in L^2(\mc Y)$
\begin{eqnarray*}
	\| \Cov(X_1,Y_1) - \Cov(X_2,Y_2) \|
	& = &
	\| \bbE[X_1 \otimes Y_1] - \bbE[X_2 \otimes Y_2]\|\\
	& \leq &
	\bbE[ \| (X_1-X_2) \otimes Y_1\| + \|X_2 \otimes (Y_1-Y_2)\| ] \\
	& = &
	\bbE[ \|X_1-X_2\| \; \|Y_1\| ] +\bbE[ \|X_2\| \; \|Y_1-Y_2\| ] \\
	& \leq &
	(\|Y_1\|_{L^2} + \|X_2\|_{L^2})  \; (\|X_1-X_2\|_{L^2} + \|Y_1-Y_2\|_{L^2} ),
\end{eqnarray*}
where we have used Jensen's and the triangle inequality in the second line and the Cauchy-Schwartz inequality in the last line.
Since $\Cov(X,Y) = \Cov(X - \bbE[X],Y - \bbE[Y])$ and $\|X - \bbE[X]\|_{L^2} \leq \|X\|_{L^2}$ the above estimate holds also for non-zero-mean RVs.
Thus, we get
\[
	\|\Cov(U,Z) - \Cov(U_J,Z_J)\| \leq ( \|U\|_{L^2} + \|Z\|_{L^2})  \; (\|U-U_J\|_{L^2} + \|Z-Z_J\|_{L^2} )
\]
and
\[
	\|\Cov(Z) - \Cov(Z_J)\| \leq 4\|Z\|_{L^2} \; \|Z-Z_J\|_{L^2},
\]
due to $\|Z_J\|_{L^2} \leq \|Z\|_{L^2}$.
Now consider again the assumed monotone and exhaustive sequence $(J_n)_{n\in\bbN}$ and recall that, by taking a sufficiently large $n$, the error $\|U - U_{J_n}\|_{L^2}$ and $\|Z - Z_{J_n}\|_{L^2}$ can be made arbitrarily small.
Thus, also $\|\Cov(Z) - \Cov(Z_J)\|$ will tend to zero as $n\to \infty$.
We now apply now the continuity of the matrix inverses of $\Cov(Z), \Cov(Z_{J_n}) \in \bbR^{d\times d}$.
Specifically, if $n$ is sufficiently large that
\[
	\|\Cov(Z) - \Cov(Z_{J_n})\| < \frac 1{2 \|\Cov^{-1}(Z)\|},
\]
then there holds
\[
	\|\Cov^{-1}(Z) - \Cov^{-1}(Z_{J_n})\| \leq  2 \|\Cov^{-1}(Z)\|^2  \|\Cov(Z) - \Cov(Z_{J_n})\| 
\]
(see \cite[Sect. 5.8]{HornJohnson1990}).
Summing up all previous estimates, we obtain
\begin{eqnarray*}
	\|K - K_{J_n}\|
	& \leq &
	C_1 (\|U-U_{J_n}\|_{L^2} + \|Z-Z_{J_n}\|_{L^2} )
	+
	C_2 \|Z-Z_{J_n}\|_{L^2},
\end{eqnarray*}
with $C_1 = \|\Cov^{-1}(Z)\| (\|U\|_{L^2} + \|Z\|_{L^2})$ and $C_2 =  8 \|U\|_{L^2} \; \|\Cov^{-1}(Z)\|^2 \; \|Z\|_{L^2}^2$ where we have used
\[
	\|\Cov(U_J,Z_J)\| \leq \|U_J\|_{L^2} \, \|Z_J\|_{L^2} \leq \|U\|_{L^2} \|Z\|_{L^2}
\]
to obtain $C_2$.
Finally, we arrive at
\begin{eqnarray*}
	\|U^a - U^a_{J_n}\|_{L^2}
	&  \leq &
	\|U - U_{J_n}\|_{L^2} +  (|z|+ \|Z\|_{L^2}) \|K - K_{J_n}\| + \|K\| \; \|Z - Z_{J_n}\|_{L^2}\\
	& \leq & C (\|U-U_{J_n}\|_{L^2} + \|Z-Z_{J_n}\|_{L^2}),
\end{eqnarray*}
with $C = 1 + \|K\| + |z| + \|Z\|_{L^2} + C_1 + C_2$, and the assertion follows. \hfill
\end{proof}

\begin{remark}
~
Since for many applications evaluating the forward map $G$ corresponds to solving a differential or integral equation, an additional error arises due to numerical approximations $G_h$ of $G$.
This error affects the filters again by instead sampling or computing chaos coefficients of $Z_h = G_h(U) + \varepsilon$ than $Z$.
We neglect this error in our analysis since it is bayond the scope of this work.
However, if $G$ is the solution operator for differential equations, we expect that \eqref{equ:G_bound_assum} could be verified in many cases, such as for elliptic boundary value problems with $U$ a random diffusion coefficient or source term.
\end{remark}

A first convergence analysis for the EnKF when the sample size tends to infinity was carried out in \cite{MandelEtAl2011}.
There the authors considered finite-dimensional linear systems, and their main goal was to show the convergence of the ensemble mean and covariance to the true posterior mean and covariance in $L^p(\bbR^n)$ and $L^p(\bbR^{n\times n})$, respectively.
We will now show the $\bbP$-almost sure convergence of the empirical distribution $\tilde \mu^a_M$ defined by the EnKF analysis ensemble to the distribution of $U^a \sim \mu^a$.
\begin{theorem} 
Given the model \eqref{equ:ip_bayes} under Assumptions  
\ref{assum:A1}, 
\ref{assum:A2} and 
\ref{assum:A3}, 
let $(u_1^a,\ldots,u_M^a)$ denote the analysis ensemble resulting from the EnKF and $\tilde \mu_M^a$ the associated empirical measure \eqref{equ:EnKF_measure}.
Further, let $\mu^a$ denote the push-forward measure of the analysis variable $U^a$.
Then, for any $f: \mc X \to \mc Y$ which satisfies
\[
	\|f(u) - f(v)\|_{\mc Y} \leq C (1+\|u\|_{\mc X}+\|v\|_{\mc X}) \; \|u-v\|_{\mc X}
	\qquad \forall u,v \in \mc X,
\]
where $\mc Y$ is any separable Hilbert space, there holds 
\[
	\lim_{M \to \infty} \int_{\mc X} f(u)\, \tilde \mu_M^a(\D u)
	= 
	\int_{\mc X} f(u)\, \mu^a(\D u) \qquad \bbP\text{-a.s.}
\]
This implies, in particular,
\[
	\lim_{M \to \infty} \frac 1M \sum_{j=1}^M u^a_j  = \bbE[U^a] 
	\quad \text{ and } \quad 
	\lim_{M \to \infty} \frac 1M \sum_{i,j=1}^M u^a_i \otimes u^a_j = \Cov(U^a) \qquad \bbP\text{-a.s.}
\]
as well as
\[
	\bbP\left( \lim_{M\to \infty} d_W(\tilde \mu_M^a, \mu^a) = 0 \right) = 1,
\]
\end{theorem}
\begin{proof}
We denote by $U_i$ and $\varepsilon_i$, $i\in\bbN$, i.i.d. RV such that $(U_i,\varepsilon_i) \sim \mu_0 \otimes \nu_\varepsilon$.
Further, we define
\[
	U_i^a := U_i + K(z-Z_i), \qquad K = \Cov(U_1,Z_1)\Cov^{-1}(Z_1),
\]
where $Z_i := G(U_i) + \varepsilon_i$, and
\[
	X_{M,i}^a := U_i + K_M (z-Z_i), \qquad K_M = \Cov(\bs U_M,\bs Z_M)\Cov^{-1}(\bs Z_M),
\]
where $\Cov(\bs U_M,\bs Z_M)$ and $\Cov(\bs Z_M)$ are empirical covariances, e.g.,
\[
	\Cov(\boldsymbol{U}_M, \boldsymbol{Z}_M) = \frac 1{M-1} \sum_{i=1}^M (U_i - \bar{\bs U}_M) \otimes (Z_i - \bar{\bs Z}_M)
\]
with $\bar{\bs U}_M = \frac 1M(U_1+\cdots+U_M)$ and  $\bar{\bs Z}_M = \frac 1M (Z_1+\cdots+Z_M)$.
Note that $(X^a_1,\ldots,X^a_M)$ represents the random analysis ensemble of the EnKF and that $U^a_i \sim \mu^a$ i.i.d.
For any function $f: \mc X \to \mc Y$ which fulfills the assumptions stated in the theorem, we have
\begin{eqnarray*}
	\frac 1M \sum_{i=1}^M f(X^a_{M,i})
	& = &
	\frac 1M \sum_{i=1}^M f(X^a_{M,i}) - f(U^a_{i}) + \frac 1M \sum_{i=1}^M f(U^a_{i})
\end{eqnarray*}
where there holds
\[
	\lim_{M\to \infty} \frac 1M \sum_{i=1}^M f(U^a_{i}) = \int_{\mc X} f(u)\, \mu^a(\D u) \qquad \bbP\text{-a.s.}
\]
due to the strong law of large numbers (SLLN) \cite{PadgettTaylor1973}.
Hence, we need only ensure that
\begin{eqnarray*}
	\left\| \frac 1M \sum_{i=1}^M f(X^a_{M,i}) - f(U^a_{i}) \right\|
	& \leq &
	\frac 1M \sum_{i=1}^M C (1+\|U^a_{i}\| + \|X^a_{M,i}\|) \; \|X^a_{M,i} - U^a_{i}\| \\
	& \leq &
	\left( \frac {C}{M} \sum_{i=1}^M (1+\|U^a_{i}\| + \|X^a_{M,i}\|)^2\right)^{1/2} \;  \left(\frac {C}{M} \; \sum_{i=1}^M \|X^a_{M,i} - U^a_{i}\|^2\right)^{1/2}
\end{eqnarray*}
converges $\bbP$-a.s. to 0 as $M\to\infty$ to prove the first statement.
We estimate
\[
	\|X^a_{M,i} - U^a_i\| \leq \|K - K_M\| \|z - Z_i\| \qquad \forall i\in\bbN,
\]
where we can further split
\begin{eqnarray*}
	K - K_M & = & \big( \Cov(U,Z) - \Cov(\bs U_M,\bs Z_M) \big) \Cov^{-1}(Z)\\
	& & \quad  + \Cov(\bs U_M,\bs Z_M) \, \left(\Cov^{-1}(Z) -  \Cov^{-1}(\bs Z_M)\right).
\end{eqnarray*}
Next, we recall that the empirical covariance converges $\bbP$-almost surely to the true covariance which follows easily (see \cite[Satz 3.14]{MuellerGronbachEtAl2012} for the scalar case) by writing
\[
	\Cov(\bs U_M,\bs Z_M)
	=
	\frac 1{M-1} \sum_{i=1}^M (U_i - \bbE[U]) \otimes (Z_i - \bbE[Z])
	- \frac M{M-1} (\bar{\bs U}_M - \bbE[U]) \otimes (\bar{\bs Z}_M - \bbE[Z]).
\]
Then by the SLLN we get
\[
	\frac 1{M-1} \sum_{i=1}^M (U_i - \bbE[U]) \otimes (Z_i - \bbE[Z])  \; \xrightarrow{M\to \infty} \; \bbE[(U - \bbE[U]) \otimes (Z - \bbE[Z]),
\]
and $\frac {M}{M-1}(\bar{\bs U}_M - \bbE[U]) \otimes (\bar{\bs Z}_M - \bbE[Z]) \xrightarrow{M\to \infty} 0$ $\bbP$-almost surely.
Thus, we have
\[
	 \Cov(U,Z) - \Cov(\bs U_M,\bs Z_M) \xrightarrow{M\to \infty} 0 \quad \text{ and } \quad \Cov(Z) -  \Cov(\bs Z_M) \xrightarrow{M\to \infty} 0
\]
$\bbP$-almost surely. Since the matrix inverse is a continuous mapping there also follows
\[
	\Cov^{-1}(Z) -  \Cov^{-1}(\bs Z_M) \xrightarrow{M\to \infty} 0 \quad \bbP\text{-a.s.}
\]
and hence, $K-K_M \to 0$ as $M\to \infty$ $\bbP$-almost surely.
We thus have for $p\in [1,2]$ $\bbP$-a.s.
\[
	\lim_{M\to \infty} X^a_{M,i} = U^a_i \quad \forall i\in\bbN \qquad \text{and} \qquad \lim_{M \to \infty} \frac 1M \sum_{i=1}^M \|X^a_{M,i} - U^a_i\|^p = 0,
\]
since by the SLLN $\frac 1M \sum_{i=1}^M \|z - Z_i\|^p$ will tend to $\bbE[\|z-Z\|^p]$ $\bbP$-almost surely.
Moreover, there holds
\[
	(1+\|U^a_{i}\| + \|X^a_{M,i}\|)^2 \leq (1+2\|U^a_{i}\| + \|X^a_{M,i}-U^a_{i}\|)^2 \leq 2(1+2\|U^a_{i}\|)^2 + \|X^a_{M,i}-U^a_{i}\|^2
\]
which yields, again by the SLLN and the above arguments,
\[
\frac {1}{M} \sum_{i=1}^M (1+\|U^a_{i}\| + \|X^a_{M,i}\|)^2 \leq \frac {1}{M} \sum_{i=1}^M (1+2\|U^a_{i}\|)^2 + \frac {1}{M} \sum_{i=1}^M \|X^a_{M,i}-U^a_i\|^2 \rightarrow \bbE[(1+2\|U^a\|)^2]
\]
as $M\to\infty$ $\bbP$-a.s.
We thus finally obtain
\[
	\left\| \frac 1M \sum_{i=1}^M f(X^a_{M,i}) - f(U^a_{i}) \right\| \xrightarrow{M\to\infty} 0 \qquad \bbP\text{-a.s.},
\]
proving the first statement of the theorem.
The remaining three then follow immediately.
\hfill
\end{proof}

%
%
\section{Bayesian Interpretation of Generalized Kalman Filters} \label{sec:GKF_Bayes}

In the previous section we have characterized the limit of the EnKF and PCKF approximations for increasing sample size or polynomial basis, respectively.
We now investigate how this limit, the analysis variable $U^a$, may be understood in the context of Bayesian inverse problems.
By analyzing the properties of this RV we are able to characterize those of the approximations provided by the two Kalman filtering methods.
In particular, we show that these do not, in general, solve the nonlinear Bayesian inverse problem, nor can they be even justified as approximations to its solution.
They are, rather, related to a linear approximation of the Bayes estimator $\hat\phi_\mathrm{CM}$ and its estimation error.

\subsection{The Linear Conditional Mean} \label{sec:LCM}

The quantity known in classical statistics as the best linear unbiased estimator (BLUE) corresponds in the Bayesian setting to the \emph{linear posterior mean estimator} $\hat \phi_\mathrm{LCM}$ defined as
\begin{equation} \label{defLCM}
	\hat \phi_\mathrm{LCM} = \argmin_{\phi \in \mc P_1(\bbR^d;\mc X)} \ev{ \|U - \phi(Z)\|^2},
\end{equation}
where $\mc P_1(\bbR^d;\mc X) = \{\phi: \phi(z) = b + Az \text{ with }b\in \mc X, A \in \mc L(\bbR^d, \mc X)\}$ denotes the set of all linear mappings from $\bbR^d$ to $\mc X$.
Moreover, we refer to the RV $\hat \phi_\mathrm{LCM}(Z)$ as the \emph{linear conditional mean}.
Recall that the conditional mean $\hat\phi_\mathrm{CM}(Z) = \bbE[U|Z]$ is the best approximation of $U$ in $L^2(\Omega, \sigma(Z), \bbP; \mc X)$ w.r.t.\ the $L^2(\mc X)$-norm.
Thus $\hat \phi_\mathrm{LCM}(Z)$ can be seen as the best approximation of $U$ in the subspace $\mc P_1(Z; \mc X) \subset L^2(\Omega,\sigma(Z), \bbP; \mc X)$, where  $\mc P_1(Z; \mc X)$ is short for $\mc P_1(\bbR^d; \mc X) \circ Z =\{\phi(Z), \phi \in \mc P_1(\bbR^d; \mc X)\}$.

\begin{lemma} \label{propo:LCM0}
The linear conditional mean as defined in \eqref{defLCM} is given by
\[
	\hat \phi_\mathrm{LCM}(z) = \ev{U} + \Cov(U,Z)\Cov(Z)^{-1}(z - \ev{Z}).
\]
\end{lemma}
\begin{proof}
The assertion follows by verifying that
\[
	\phi(Z) = \ev{U} + K(Z - \ev{Z}), \qquad K = \Cov(U,Z)\Cov(Z)^{-1},
\]
coincides with the orthogonal projection of $U$ to $\mc P_1(Z;\mc X)$.
To do so, we will show that $U - \phi(Z)$ is orthogonal to $\mc P_1(Z;\mc X)$ w.r.t. the inner product in $L^2(\Omega, \mc F, \bbP; \mc X)$.
\par
Let $b\in\mc X$ and $A \in \mc L(\bbR^d, \mc X)$ be arbitrary.
Then there holds
\begin{align*}
	\bbE\left[ \langle U - \phi(Z), b + AZ \rangle \right]
	& = \underbrace{\bbE\left[ \langle U - \bbE[U], b \rangle \right]}_{=\,0} \; + \; \bbE\left[ \langle U - \bbE[U], AZ \rangle \right] \\
	& \qquad - \bbE\left[ \langle K(Z-\bbE[Z]), AZ \rangle \right] - \underbrace{\bbE\left[ \langle K(Z-\bbE[Z]), b \rangle \right]}_{=\,0}\\
	& = \bbE\left[ \langle U - \bbE[U], A(Z-\bbE[Z]) \rangle \right] - \bbE\left[ \langle K(Z-\bbE[Z]), A(Z-\bbE[Z] \rangle \right]\\
	& = \Cov(U,Z)A^* - K\Cov(Z)A^* = 0,
\end{align*}
since 
\[
	\bbE[ \langle U - \bbE[U], A\bbE[Z] \rangle] = \bbE[ \langle K(Z - \bbE[Z]), A\bbE[Z] \rangle] = 0
\]
and $\Cov(AX,BY) = A\Cov(X,Y)B^*$ for Hilbert space valued RV $X,Y$ and bounded, linear operators $A,B$.
\end{proof}

We note that Proposition \ref{propo:LCM0} fails to hold in case $\mc X$ is only a separable Banach space, since then the expectation $\bbE[U]$ and covariance $\Cov(U,Z)$ no longer minimize $\bbE[\|U - b\|^2]$, $b\in\mc X$, and $\bbE[\|U - AZ\|^2]$, $A\in\mc L(\bbR^d, \mc X)$, respectively; see also Remark \ref{rem:EX_in_BS}.

\subsection{Interpretation of the Analysis Variable} \label{sec:Inter_AV}

Lemma~\ref{propo:LCM0} immediately yields a characterization of the analysis variable $U^a$ defined in \eqref{equ:analysis_variable}.
\begin{theorem} \label{propo:LCM}
Let Assumptions  \ref{assum:A1}, \ref{assum:A2} and \ref{assum:A3} be satisfied for the model \eqref{equ:ip_bayes}.
Then for any $z\in\bbR^d$ the \emph{analysis variable} $U^a = U + K(z -Z)$, $K = \Cov(U,Z)\Cov(Z)^{-1}$, coincides with
\[
	U^a = \hat \phi_\mathrm{LCM}(z) + (U - \hat \phi_\mathrm{LCM}(Z)).
\]
In particular, there holds
\[
	\ev{U^a} = \hat \phi_\mathrm{LCM}(z)
	\quad \text{ and } \quad
	\Cov(U^a) = \Cov(U) - K \Cov(Z, U).
\]
\end{theorem}
We summarize the consequences of Theorem~\ref{propo:LCM} as follows:\begin{itemize}
\item
The analysis variable $U^a$, to which the EnKF and the PCKF provide approximations, is the sum of a Bayes estimate $\hat \phi_\mathrm{LCM}(z)$ and the prior error $U - \hat \phi_\mathrm{LCM}(Z)$ of the corresponding Bayes estimator $\hat \phi_\mathrm{LCM}$.
\item
The mean of the EnKF analysis ensemble or PCKF analysis vector provide approximations to the linear posterior mean estimate.
How far the latter deviates from the true posterior mean depends on the model and observation $z$.
\item
The covariance approximated by the empirical covariance of the EnKF analysis ensemble,
as well as that of the PCKF analysis vector, is independent of the actual observational data $z\in\bbR^d$.
It therefore constitutes a prior rather than a posterior measure of uncertainty.
\item
In particular, the randomness in $U^a$ is entirely determined by the prior measures $\mu_0$ and $\nu_\varepsilon$.
Only the location, i.e., the mean, of $U^a$ is influenced by the observation data $z$; the randomness of $U^a$ is independent of $z$ and determined only by the projection error $U - \hat \phi_\mathrm{LCM}(Z)$ w.r.t. the prior measures.
\item
In view of the last two items, the analysis variable $U^a$, and therefore the EnKF analysis ensemble or the result of the PCKF, are in general not distributed according to the posterior measure $\mu^z$.
Moreover, the difference between $\mu^z$ and the distribution of $U^a$ depends on the data $z$ and can become quite large for nonlinear problems, see Example \ref{exam:3}.

\end{itemize}

\begin{remark}
In particular the second and third item above explain the observations made in \cite{LawStuart2012} that ``[...] (i) with appropriate parameter choices, approximate filters can perform well in reproducing the mean of the desired probability distribution, (ii) they do not perform as well in reproducing the covariance [...] ''.
\end{remark}


We illustrate the conceptual difference between the distribution of the analysis variable $U^a$ and the posterior measure $\mu^z$ with a simple yet striking example.

\begin{example} \label{exam:3}
We consider $U \sim N(0,1)$, $\varepsilon \sim N(0,\sigma^2)$ and $G(u) = u^2$.
Given data $z \in \bbR$, the posterior measure, obtained from Bayes' rule for the densities, is
\[
	\mu^z(\D u) = C \exp\left(-\frac {\sigma^2u^2 + (z-u^2)^2}{2\sigma^2}\right) \, \D u.
\]
Due to the symmetry of $\mu^z$ we have $\hat u_\mathrm{CM} = \int_{\mc X} u \, \mu^z(\D u) = 0$ for any $z\in \bbR^d$.
Thus, $\bbE[U|Z]  \equiv 0$ and $\hat \phi_\mathrm{LCM} \equiv \hat \phi_\mathrm{CM}$.
In particular, we have $K = 0$ due to
\[
	\Cov(U,Z) = \Cov(U,U^2) = \frac 1{\sqrt{2\pi}} \int_\bbR u (u^2 - 1) \E^{-u^2/2} \D u = 0,
\]
which in turn yields $U^a = U \sim N(0,1)$. Thus, the analysis variable is distributed according to the prior measure.
This is not surprising as, by definition, its mean is the best linear approximation to  the posterior mean according to $\mu^z$ and its fluctuation is simply the prior estimation error $U - \hat \phi_\mathrm{LCM}(Z) = U - 0 =U$.
This illustrates that $U^a$ is suited for approximating the posterior mean, but not appropriate as a method for uncertainty quantification for the nonlinear inverse problem.
As displayed in Figure \ref{fig:exam1}, the distribution of $U^a$ can be markedly different from the true posterior distribution.
\begin{figure}
\centering \resizebox{10cm}{!}{\includegraphics{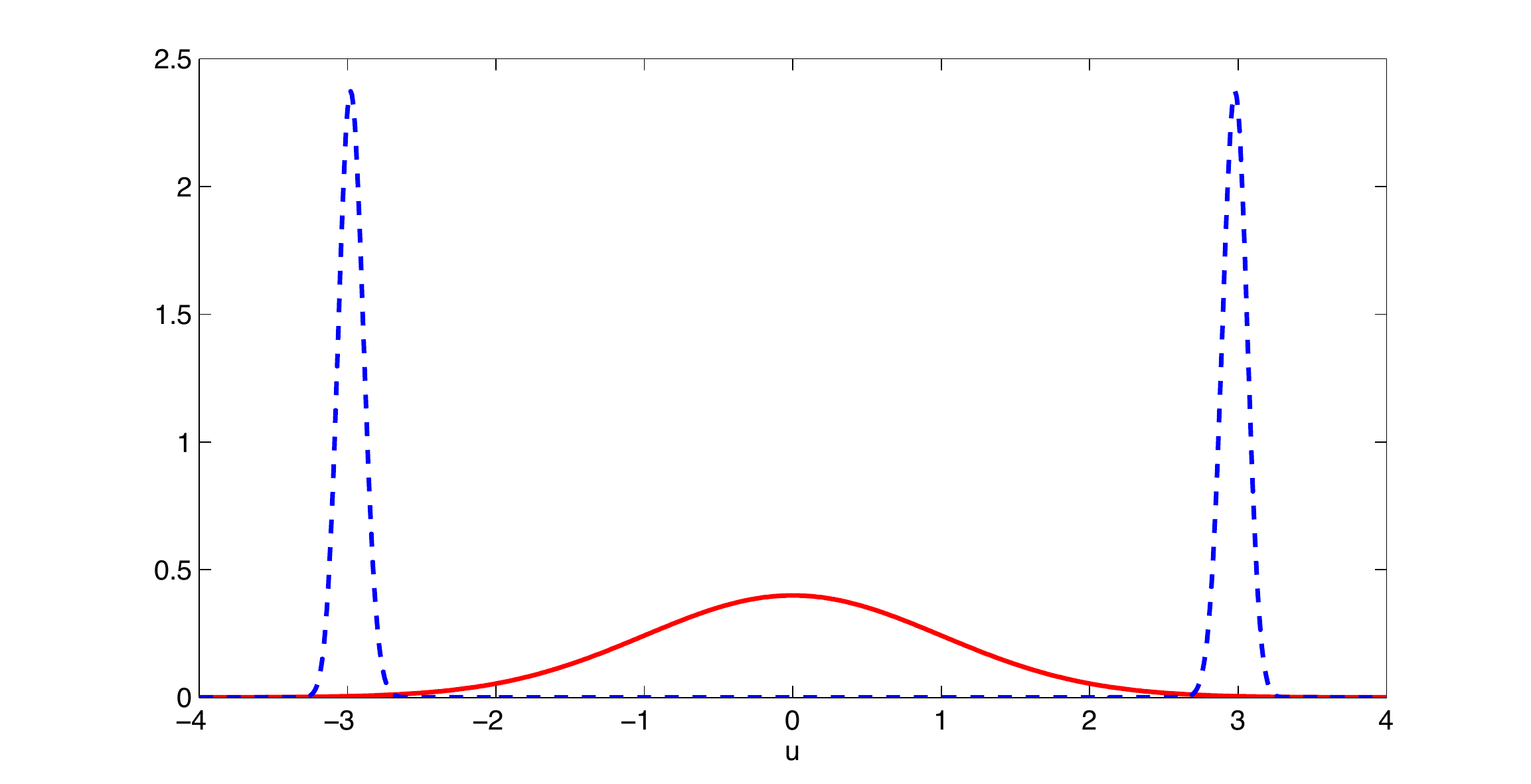}}
\caption{Density of the posterior $\mu^z$ (\textit{dashed, blue line}) and the probability density of the analysis variable $U^a$ (\textit{solid, red line}) for $z=9$ and $\sigma = 0.5$.}
\label{fig:exam1}
\end{figure}
\end{example}


\section{Numerical Examples} \label{sec:Exams}
To illustrate the application of the EnKF and PCKF to simple Bayesian inverse problems, we consider in the following a one-dimensional elliptic
boundary value problem and a time-dependent RLC circuit model.

\subsection{1D Elliptic Boundary Value Problem} \label{sec:Exam1}

Let $D=[0,1]$ and
\begin{equation} \label{equ:1D_PDE}
	-\frac {\D}{\D x} \left( \exp(u_1) \,\frac {\D}{\D x}p(x) \right) = f(x), \qquad p(0) = p_0, \quad p(1) = u_2,
\end{equation}
be given where $u = (u_1,u_2)$ are unknown parameters.
The solution of \eqref{equ:1D_PDE} is
\begin{equation} \label{equ:1D_PDE_sol}
	p(x) = p_0 + (u_2 - p_0)x + \exp(-u_1) \left( S_x(F) - S_1(F)\,x \right),
\end{equation}
where $S_x(g) := \int_0^x g(y)\,\D y$ and $F(x) = S_x(f) = \int_0^x f(y)\, \D y$.
For simplicity we choose $f\equiv 1$, $p_0 = 0$ in the following and assume noisy measurements have been made of $p$ at $x_1 = 0.25$ and $x_2 = 0.75$ with values $z=(27.5, 79.7)$.
We seek to infer $u$ based on this data and on a priori information modelled by $(u_1, u_2) \sim N(0,1) \otimes \text{Uni}(90,110)$, where $\text{Uni}(a,b)$ denotes the uniform distribution on the interval $[a,b]$.
Thus the forward map here is $G(u) = (p(x_1), p(x_2))$, where $p$ is given in \eqref{equ:1D_PDE_sol} with $f\equiv 1$ and $p_0 = 0$.
As the model for the measurement noise we take $\varepsilon \sim N(0, 0.01\, I_2)$. \par
In Figure~\ref{fig:exam2} we show the prior and posterior densities as well as $1000$ ensemble members of the initial and analysis ensemble obtained by the EnKF.
A total ensemble size of $M=10^5$ was chosen in order to reduce the sampling error to a negligible level.
It can be seen, however, that the analysis EnKF-ensemble does not follow the posterior distribution, although its mean $(-2.92, 105.14)$ is quite close to the true posterior mean $(-2.65, 104.5)$ (computed by quadrature).
\begin{figure}[h]
\hfill
\begin{minipage}{0.49\textwidth}
	\includegraphics[width = \textwidth]{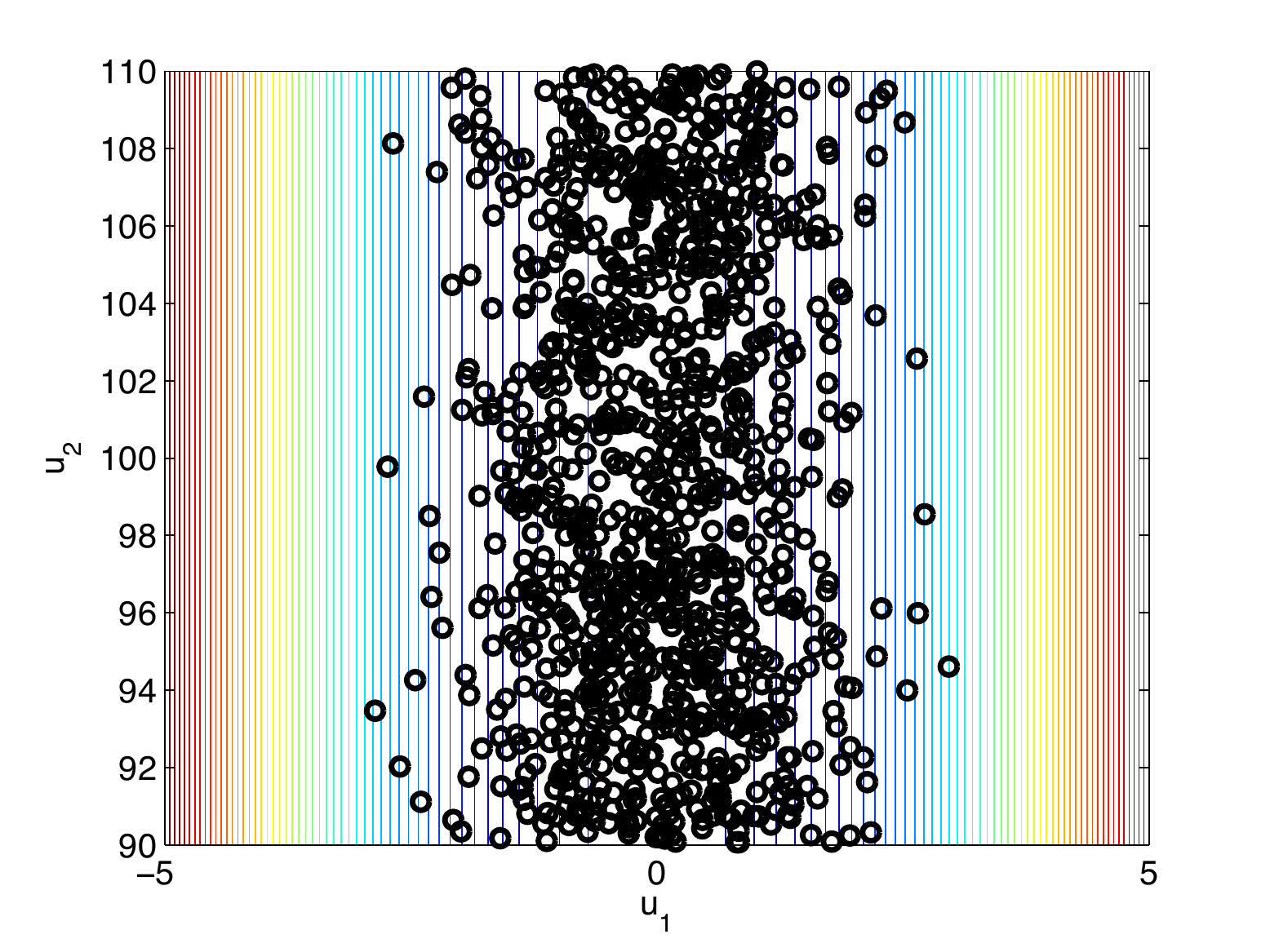}
\end{minipage}
\hfill
\begin{minipage}{0.49\textwidth}
	\includegraphics[width = \textwidth]{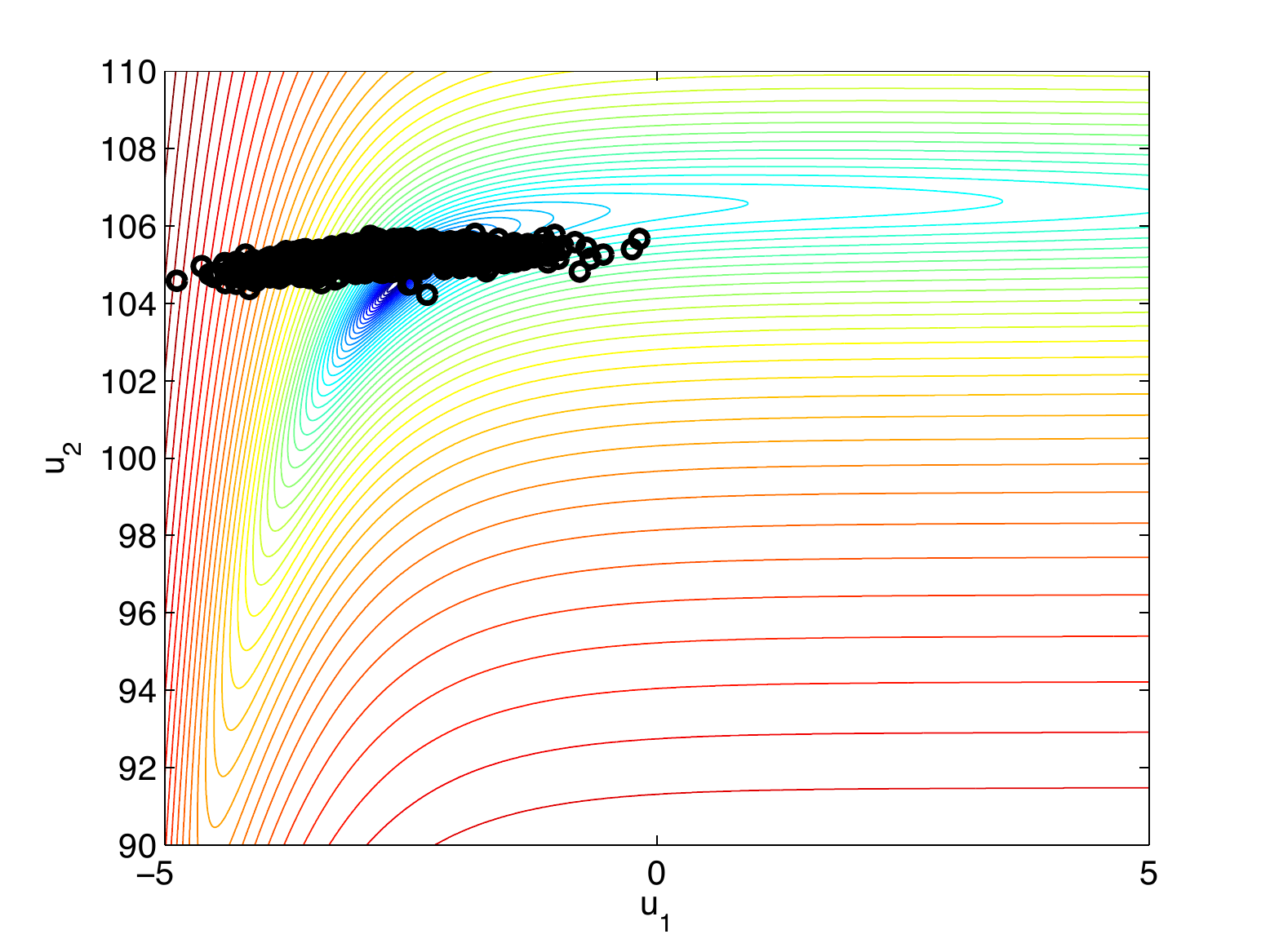}
\end{minipage}
\hfill
\caption{Left: Contour plot of the negative logarithm of the prior density and the locations of $1000$ ensemble members of the initial EnKF-ensemble.\newline 
Right: Contour plot of the logarithm of the negative logarithm of the posterior density and the locations of the updated $1,000$ ensemble members in the analysis EnKF-ensemble.}
\label{fig:exam2}
\end{figure}
To illustrate the difference between the distribution of the analysis ensemble/variable and the true posterior distribution, we present the marginal posterior distributions of $u_1$ and $u_2$ in Figure \ref{fig:exam2_marg}.
For the posterior the marginals were evaluated by quadrature, whereas for the analysis ensemble we show a relative frequency plot.

\begin{figure}[h]
\hfill
\begin{minipage}{0.49\textwidth}
	\includegraphics[width = \textwidth]{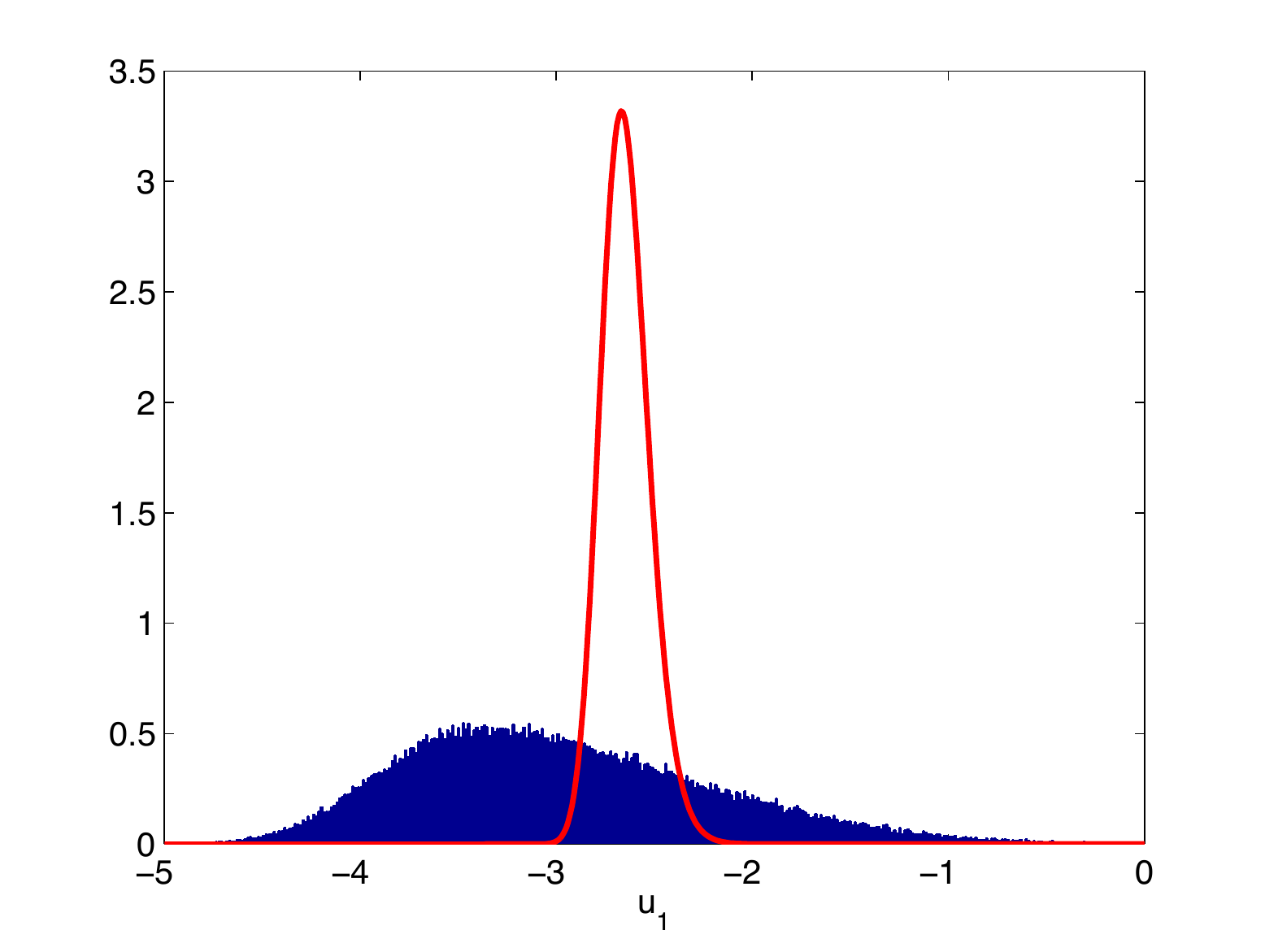}
\end{minipage}
\hfill
\begin{minipage}{0.49\textwidth}
	\includegraphics[width = \textwidth]{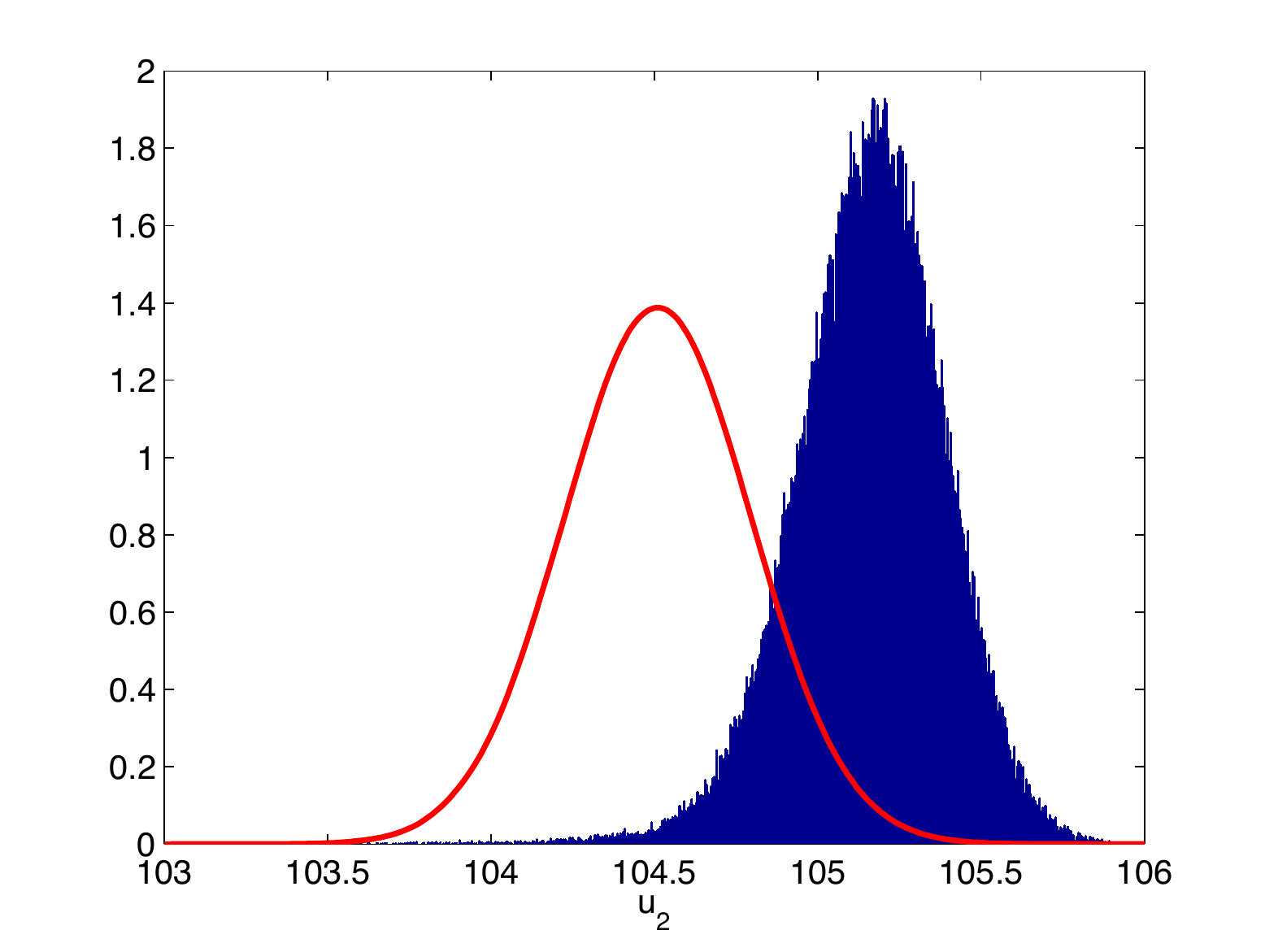}
\end{minipage}
\hfill
\caption{Posterior marginals and relative frequencies in the analysis ensemble for $u_1$ (left) and $u_2$ (right).}
\label{fig:exam2_marg}
\end{figure}
We remark that slightly changing the observational data to $\tilde z =(23.8, 71.3)$ moves the analysis ensemble
as well as the distribution of the analysis RV much closer to the true posterior, as shown in Figure~\ref{fig:exam2_obs2}.
Moreover, for these measurement values the mean of the analysis ensemble $(0.33, 94.94)$ provides a better fit to the true posterior mean $(0.33,94.94)$.

\begin{figure}[h]
\hfill
\begin{minipage}{0.32\textwidth}
	\includegraphics[width = \textwidth]{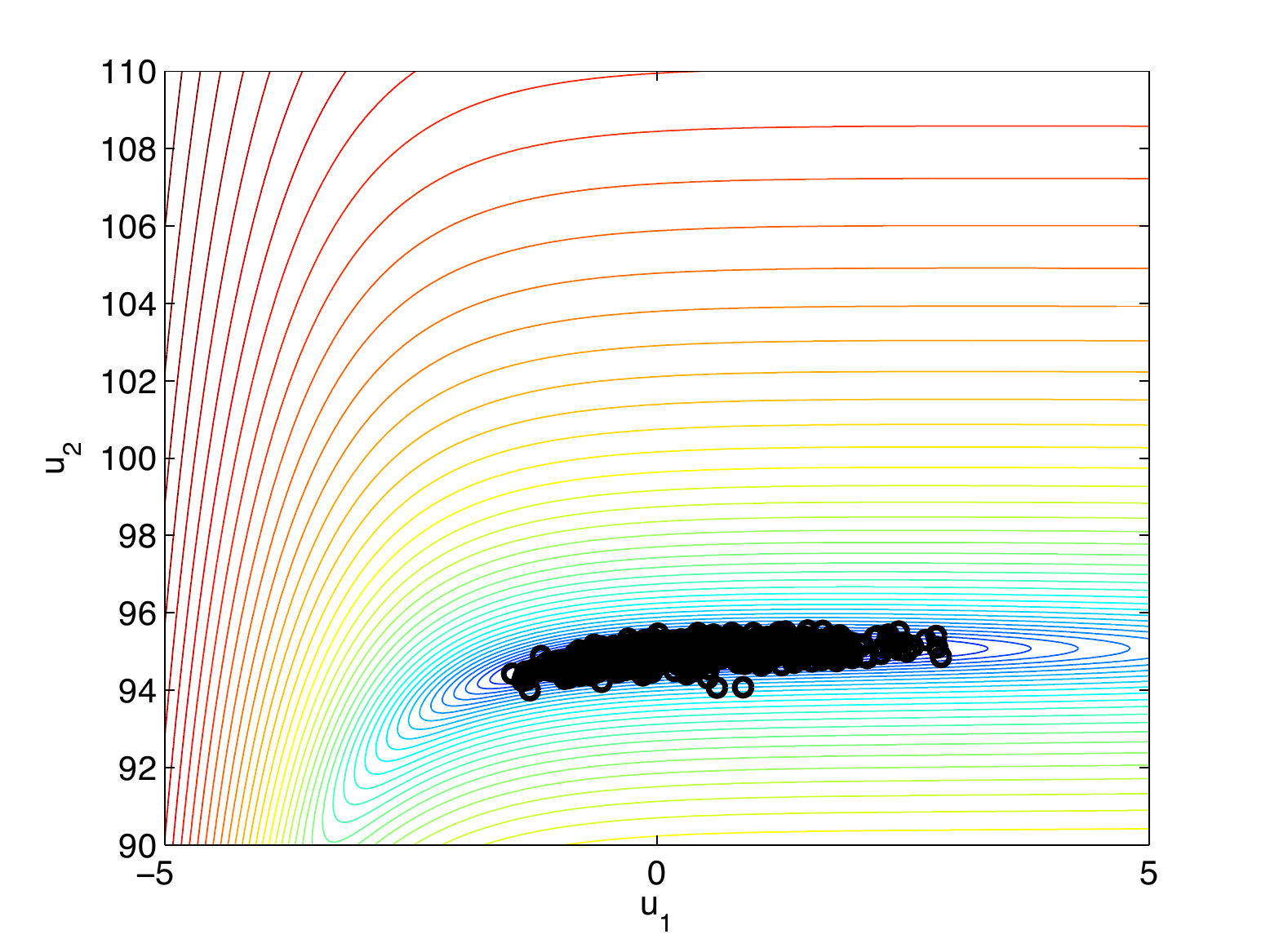}
\end{minipage}
\hfill
\begin{minipage}{0.32\textwidth}
	\includegraphics[width = \textwidth]{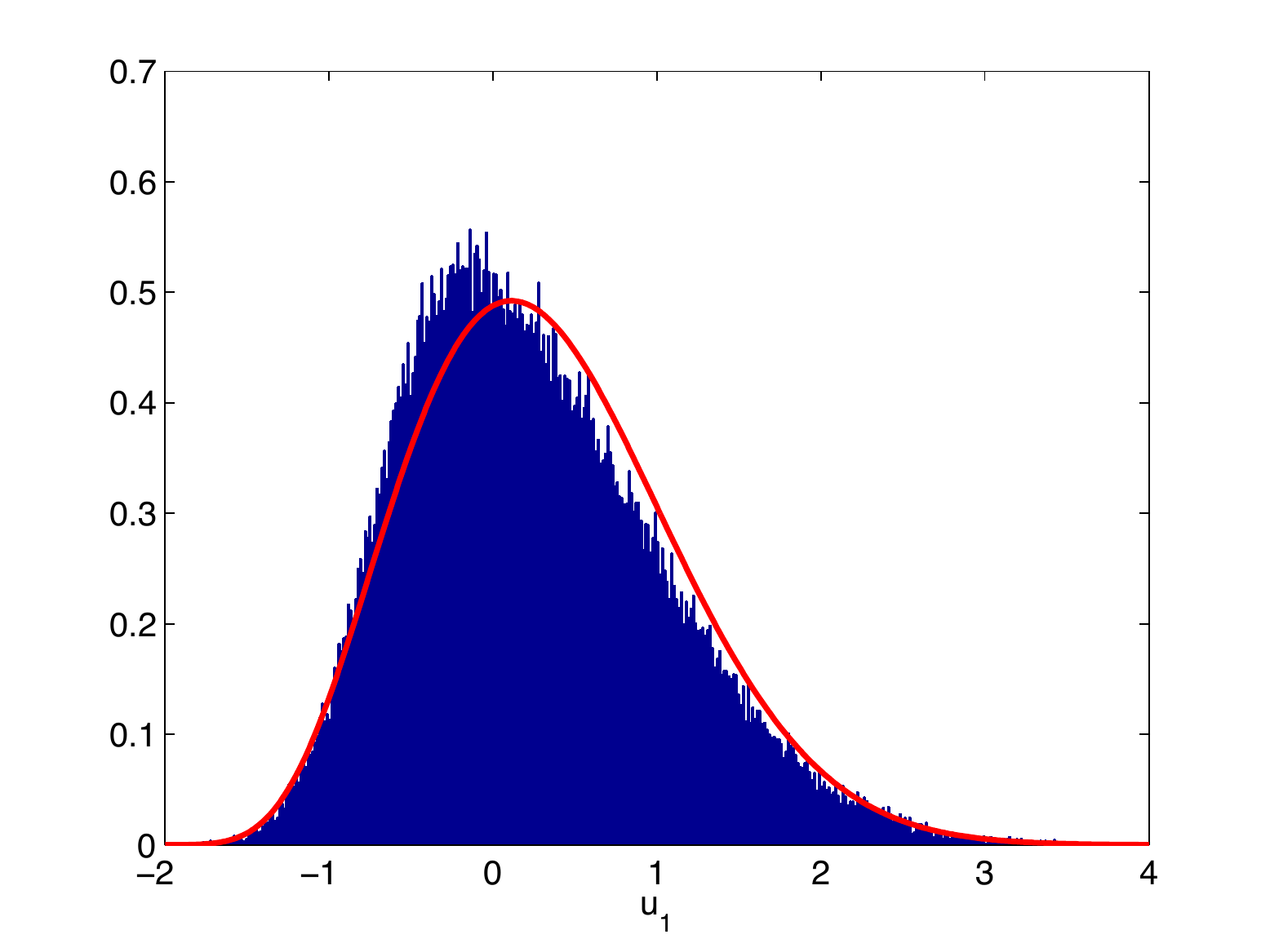}
\end{minipage}
\hfill
\begin{minipage}{0.32\textwidth}
	\includegraphics[width = \textwidth]{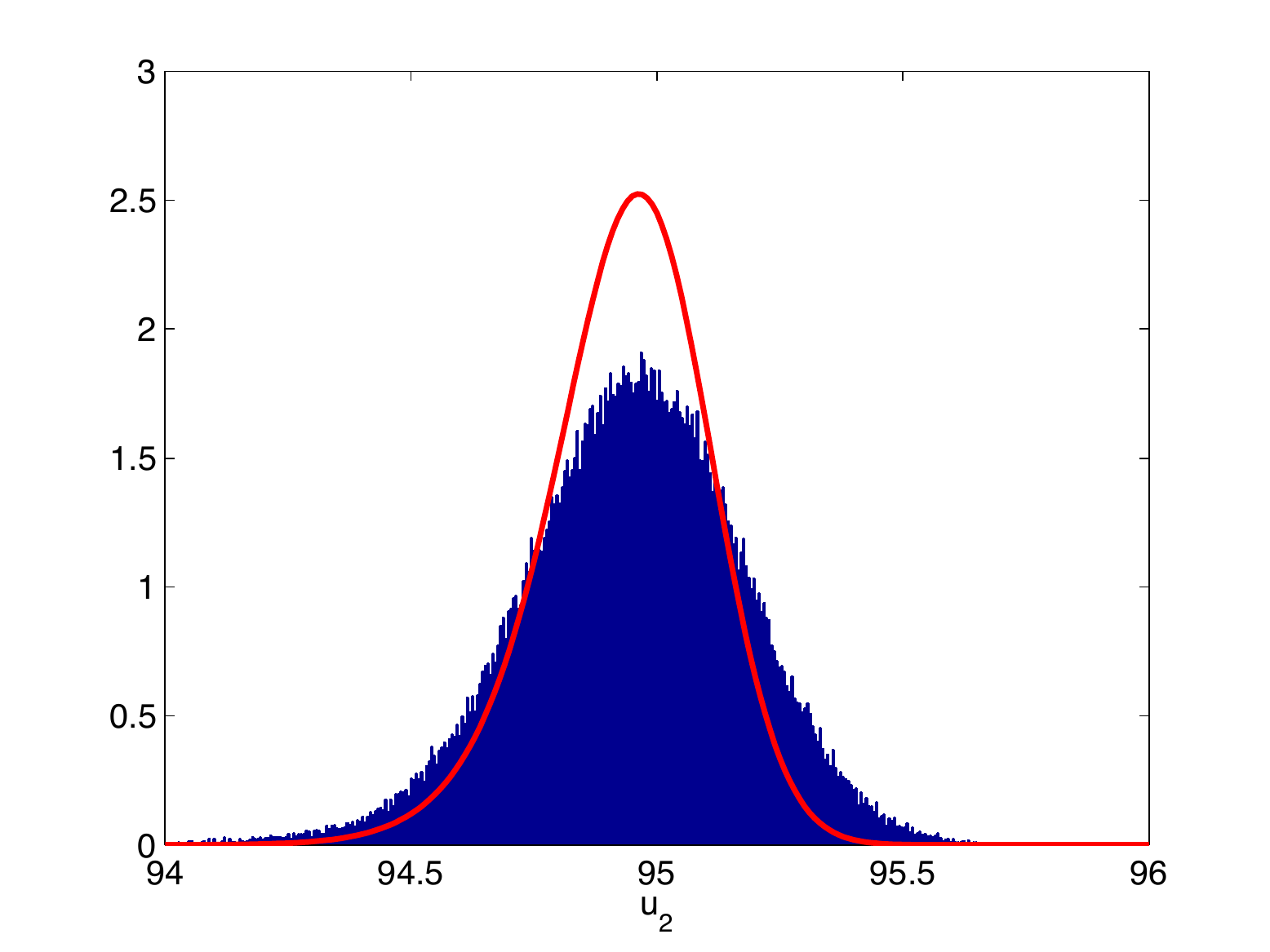}
\end{minipage}
\hfill
\caption{Left: Contours of the logarithm of the negative log posterior density and locations of $1,000$ members of the analysis EnKF-ensemble.\newline
Middle, Right: Posterior marginals and relative frequencies in the analysis ensemble for $u_1$ (middle) and $u_2$ (right).}
\label{fig:exam2_obs2}
\end{figure}

To reaffirm the fact that only the mean of the analysis variable $U^a$ depends on the actual data, we show density estimates for the marginals of $u_1$ and $u_2$ of $U^a$ in Figure \ref{fig:exam2_marg_obs} obtained from the observational data $z = (27.5, 79.7)$ (blue lines) and $\tilde z = (23.8, 71.3)$ (green lines), respectively.
The density estimates were obtained by normal kernel density estimation (KDE, in this case \textsc{Matlab}'s \texttt{ksdensity} routine) based on the resulting analysis ensembles $(\boldsymbol{u}^a_1, \boldsymbol{u}^a_2)$ and $(\tilde{ \boldsymbol{u} }^a_1, \tilde{ \boldsymbol{u} }^a_2)$ for the data sets $z$ and $\tilde z$, respectively.
We observe that the marginal distributions of the centered ensembles coincide, in agreement with Theorem~\ref{propo:LCM}.

\begin{figure}[h]
\hfill
\begin{minipage}{0.32\textwidth}
	\includegraphics[width = \textwidth]{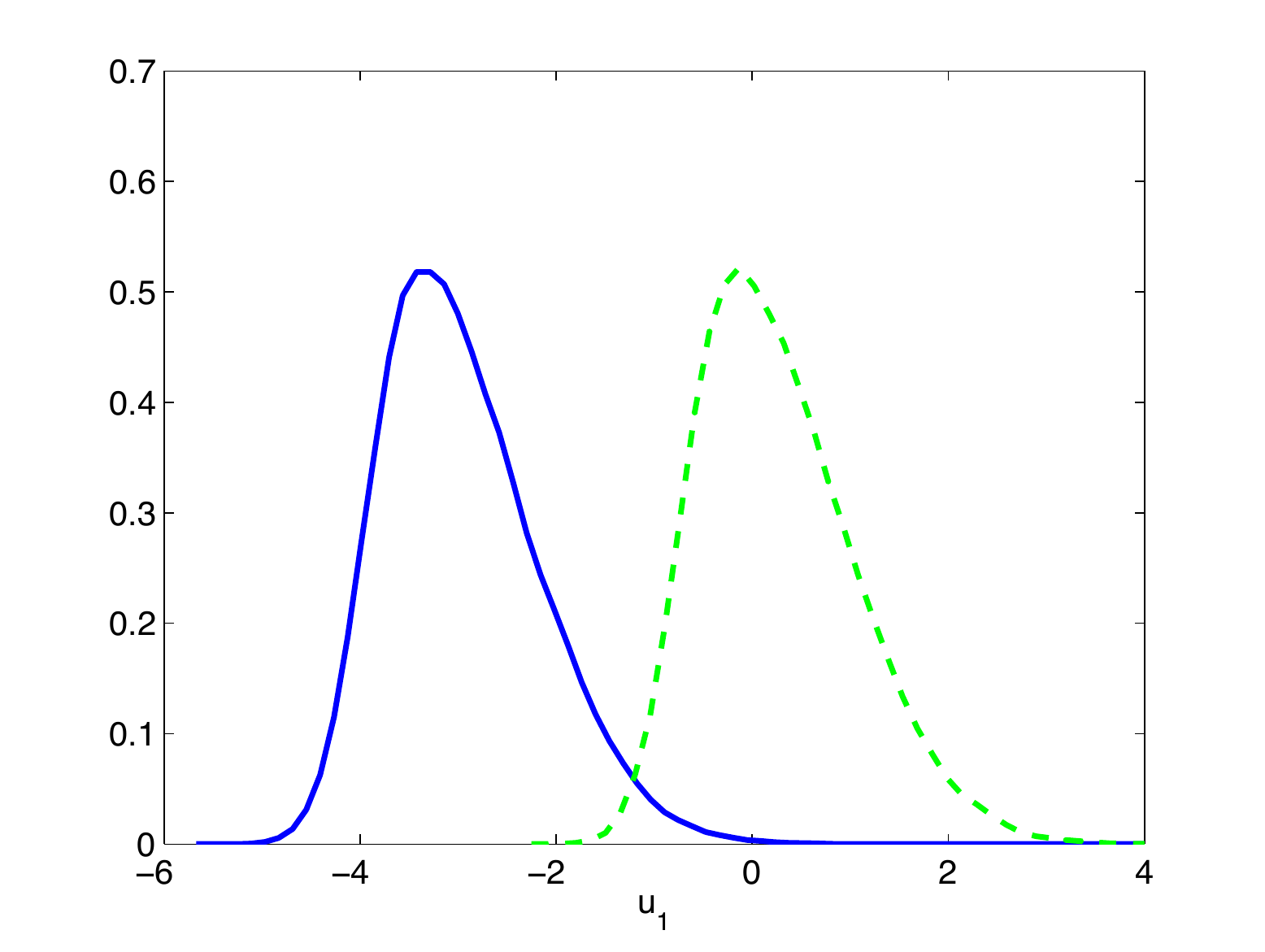}
\end{minipage}
\hfill
\begin{minipage}{0.32\textwidth}
	\includegraphics[width = \textwidth]{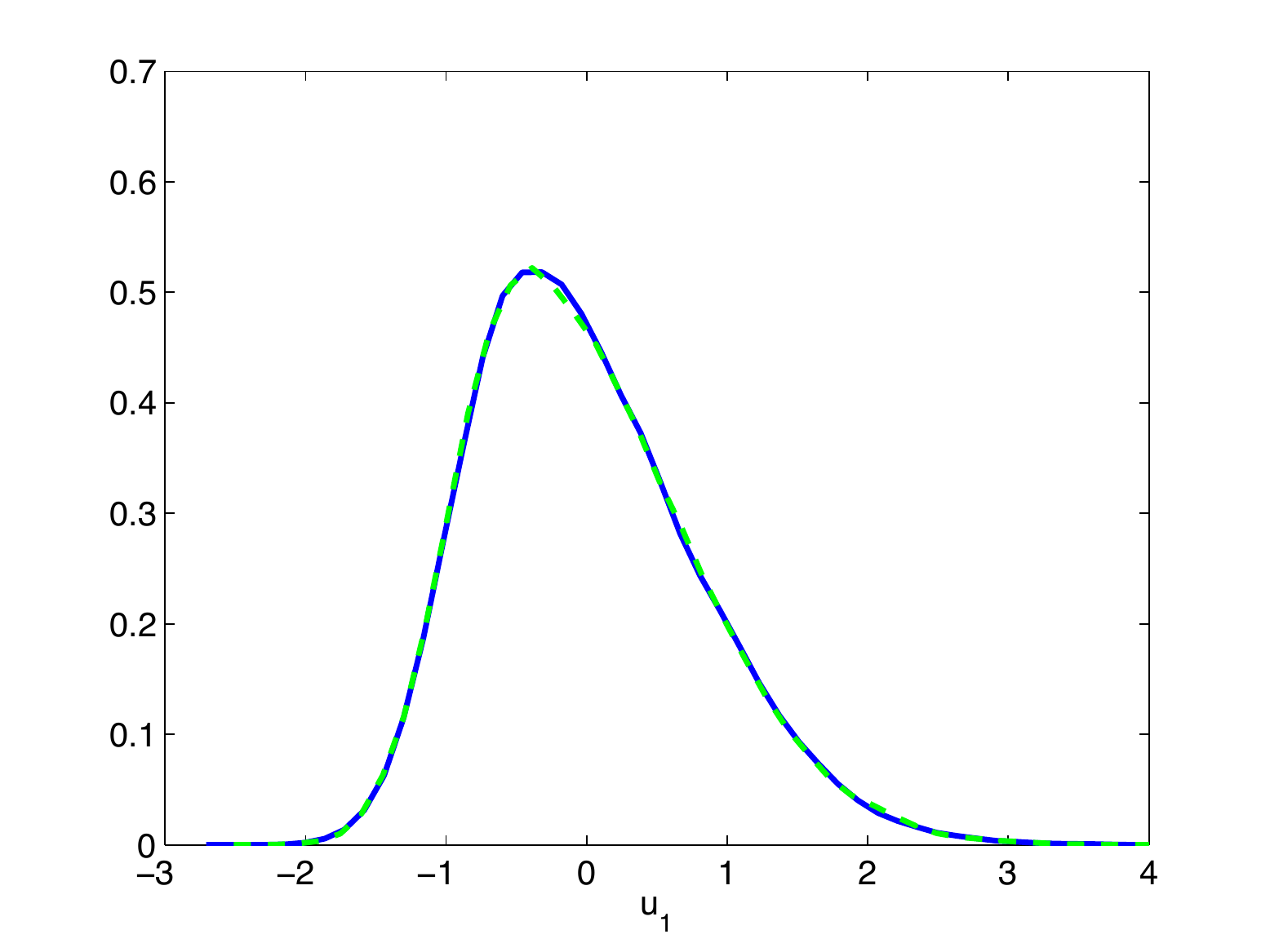}
\end{minipage}
\hfill
\begin{minipage}{0.32\textwidth}
	\includegraphics[width = \textwidth]{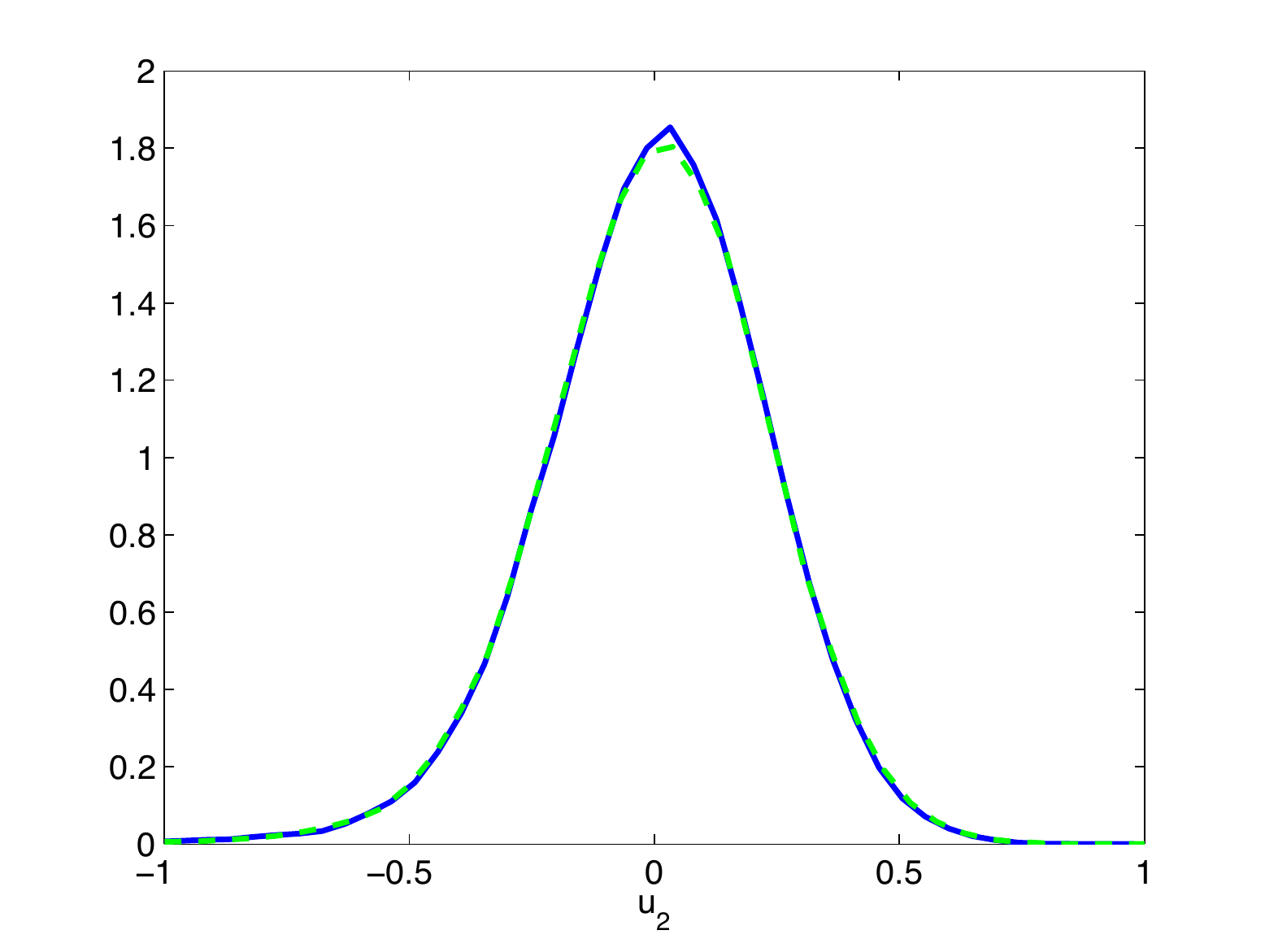}
\end{minipage}
\hfill
\caption{Left: Kernel density estimates for $\boldsymbol{u}^a_1$ (\textit{blue, solid line}) and $\tilde{ \boldsymbol{u} }^a_1$ (\textit{green, dashed line}). \newline
Middle, Right: Kernel density estimates for $\boldsymbol{u}^a_i - \bbE[\boldsymbol{u}^a_i]$  (\textit{blue, solid}) and $\tilde{ \boldsymbol{u} }^a_i - \bbE[\tilde{ \boldsymbol{u} }^a_i]$ (\textit{green, dashed}), $i=1,2$.}
\label{fig:exam2_marg_obs}
\end{figure}

In addition, whenever the prior and thus also the posterior support for $u_2$ is bounded -- as in this example -- the EnKF may yield members in the analysis ensemble which are outside this support.
This is a further consequence of Theorem \ref{propo:LCM}:
Since the analysis ensemble of the EnKF follows the distribution of the analysis variable rather than the true posterior distribution, ensemble members lying outside the posterior support can always occur whenever the support of the analysis variable is not a subset of the support of the posterior. \par
Finally, we would like to stress that, whether or not the distribution of the analysis variable is a good fit to the true posterior distribution  depends entirely on the observed data --- which can neither be controlled nor known a priori.

Applying the PCKF to this simple example problem can be done analytically.
We require four basic independent random variables $\xi_1 \sim N(0,1)$, $\xi_2 \sim \text{Uni}(0,1)$, $\xi_3 \sim N(0,1)$ and $\xi_4 \sim N(0,1)$ to define PCEs which yield random variables distributed according to the prior and error distributions:
\[
	U := 	(\xi_1, \; 90 + 20\xi_2)^\top \sim \mu_0, \qquad \varepsilon := (0.1\xi_3, \; 0.1\xi_4)^\top \sim \nu_\varepsilon.
\]
Moreover, due to \eqref{equ:1D_PDE_sol}, $G(U)$ is also available in closed form as
\[
	G(U) =
	\begin{pmatrix}
	c_{11} (90 + 20\xi_2) + c_{12} \sum_{n=0}^\infty (-1)^n \frac{\sqrt{\E}}{\sqrt{n!}} H_n(\xi_1)\\
	c_{21} (90 + 20\xi_2) + c_{22} \sum_{n=0}^\infty (-1)^n \frac{\sqrt{\E}}{\sqrt{n!}} H_n(\xi_1)
	\end{pmatrix},
\]
where $H_n$ denotes the $n$th normalized Hermite polynomial and $c_{11},c_{12},c_{21},c_{22}$ can be deduced from inserting $x=0.25$ and $x=0.75$ into \eqref{equ:1D_PDE_sol}.
Here we have used the
Hermite expansion of $\exp(-\xi)$, see also \cite[Example 2.2.7]{Ullmann2008}.
Thus, the chaos coefficient vectors of $U$ and $G(U) + \varepsilon$ w.r.t.\ the polynomials
\[
	P_{\boldsymbol{\alpha}}(\boldsymbol{\xi})
	=
	H_{\alpha_1}(\xi_1) \, L_{\alpha_2}(\xi_2) \, H_{\alpha_3}(\xi_3) \, H_{\alpha_4}(\xi_4),
	\qquad \boldsymbol\alpha\in\bbN_0^4,
\]
can be obtained explicitly where $H_\alpha$ and $L_\alpha$ denote the
normalized Hermite and Legendre polynomials of degree $\alpha$, respectively.
In particular, the nonvanishing chaos coefficients involve only the basis polynomials
\[
	P_0(\boldsymbol{\xi}) \equiv 1, \quad
	P_1(\boldsymbol{\xi}) = L_1(\xi_2), \quad
	P_2(\boldsymbol{\xi}) = H_1(\xi_3), \quad
	P_3(\boldsymbol{\xi}) = H_1(\xi_4)
\]
and $P_\alpha(\boldsymbol{\xi}) = H_{\alpha-3}(\xi_1)$ for $\alpha \geq 4$.
Arranging the two-dimensional chaos coefficients of $U$ and $G(U)$ as the column vectors of the matrices $[U], [G(U)+\varepsilon] \in \bbR^{2\times \bbN_0}$, and denoting by $\widetilde{[U]}$ the matrix $[u_1, u_2, \ldots]\in \bbR^{2\times \bbN}$ we get
\[
	K = \widetilde{[U]}\widetilde{[G(U)]}^\top \left( \widetilde{[G(U)]} \widetilde{[G(U)]}^\top
	+ 0.01 I_2 \right)^{-1}.
\]
Thus, the only numerical error incurred in applying the PCKF in this example is the truncation of the PCE.
We have carried out this calculation using a truncated PCE of length $J = 4+50$ according to the reduced basis above.
In particular, we evaluated the approximation $K_J$ to $K$ by using the truncated vector $[\mathrm P_J G(U)]$ in the formula above and then performed the update of the chaos coefficients according to \eqref{equ:PCEKF_update}.
After that $M=10^5$ samples of the resulting random variable $U_J^a$ were drawn, but since the empirical distributions were essentially indistinguishable from those obtained by the EnKF described previously, they are omitted here.

\begin{remark}
Although a detailed complexity analysis of these methods is beyond the scope of this work, we mention that the EnKF calls for $M$ evaluations of the forward map $G(u_j)$, $j=1,\ldots,M$, whereas the PCKF requires computing the chaos coefficients of $G(U)$ by, e.g., the stochastic Galerkin method.
Thus the former yields, in general, many small systems to solve, whereas the latter typically requires the solution of a large coupled system.
Moreover, we emphasize the computational savings by applying Kalman filters compared to a ``full Bayesian update'', i.e., sampling from the posterior measure by MCMC methods.
In particular, each MCMC run may require calculating many hundreds of thousands forward maps $G(u)$, e.g., for each iteration $u_j$ of the Markov chain as in the case of Metropolis-Hastings MCMC.
Hence, if one is interested in only the posterior mean as a Bayes estimate, then EnKF and PCKF provide
substantially less expensive alternatives to MCMC for its approximation by the linear posterior mean.
\end{remark}

\subsection{Dynamical System: RLC circuit} \label{sec:Exam2}

We apply the EnKF to sequential data assimilation in a simple dynamical system: a damped LC-circuit or RLC-circuit.
Denoting the initial voltage by $U_0$, the resistance by $R$, the inductance by $L$ and the capacitance by $C$, and assuming $R < 2\sqrt{LC}$, the voltage and current in the circuit can be modelled as
\begin{subequations} \label{equ:RLC}
	\begin{equation}
		U(t) = U_0 \; \E^{\delta t} \; \big(\cos(w_et) + \frac {\delta}{w_e} \sin(w_e t)\big),
	\end{equation}
	\begin{equation}
		I(t) = - \frac{U_0}{w_e L} \; \E^{\delta t} \; \sin(w_e t),
	\end{equation}
\end{subequations}
where $\delta = R/(2L)$, $w_e = \sqrt{w_0^2-\delta^2}$ and $w_0 = 1/\sqrt{LC}$.
The data assimilation setting is now as follows.
We observe the state of the system \eqref{equ:RLC} at four time points $t_n = 5n$, $n=1,\ldots,4$, where all observations $z \in \bbR^8$ are corrupted by measurement noise $\varepsilon \sim N(0, \diag(\sigma_1^2,\ldots,\sigma^2_8))$.
Here we have chosen $\sigma^2_{2n-1} = 0.1 |U(t_n)|$ and $\sigma^2_{2n} = 0.1 |I(t_n)|$ for $n=1,\ldots,4$.
We want to infer $U_0$ and $L$ based on these observations, i.e, the unknown is $u = (U_0,L)$, and we take as prior $(U_0,L) \sim N(0.5, 0.25) \otimes \text{Uni}(1,5)$.
Given observations $z \in \bbR^8$ we
compare two assimilation strategies for applying the EnKF:
\begin{itemize}
\item
\emph{Simultaneous}: We apply the EnKF to the inverse problem
\[
	z = G(u) + \varepsilon,
\]
where $G$ maps $(U_0,L)$ to the states $(U(t_1), I(t_1),\ldots,U(t_4),I(t_4)) \in \bbR^8$.
Thus, we perform one EnKF update using all the available data at once, resulting in one EnKF analysis ensemble.

\item
\emph{Sequential}: We apply the EnKF to the inverse problem
\[
	z_n = G_n(u) + \varepsilon_n, \qquad n=1,\ldots,4,
\]
where $G_n$ maps $(U_0,L)$ to the state $(U(t_n), I(t_n)) \in \bbR^2$.
In particular, we will perform four EnKF updates using at each update only the corrupted data $z_n = (U(t_n)+\varepsilon_{2n-1}, I(t_n)+\varepsilon_{2n})$.
This yields, for each update, one EnKF analysis ensemble which, in turn, serves as the initial ensemble for the next update.
\end{itemize}

Again we use two different data sets $z, \tilde z$\footnote{$z=(0.505, 0.237, 0.014, 0.096, 0.036, 0.011, -0.002, -0.003), \tilde z = (0.265$, $0.066$, $0.058$, $0.002$, $0.021$, $0.012$, $0.007$, $-0.01)$}, obtained by two realizations of $\varepsilon$ given the solution of \eqref{equ:RLC} for $U_0 = 0.75, R = 0.5, L = 1.5, C = 0.5$.

%

The resulting posteriors and EnKF analysis ensembles for the simultaneous and sequential update are presented in Fig. \ref{fig:RLC_post}.
\begin{figure}[h]
\includegraphics[width = \textwidth]{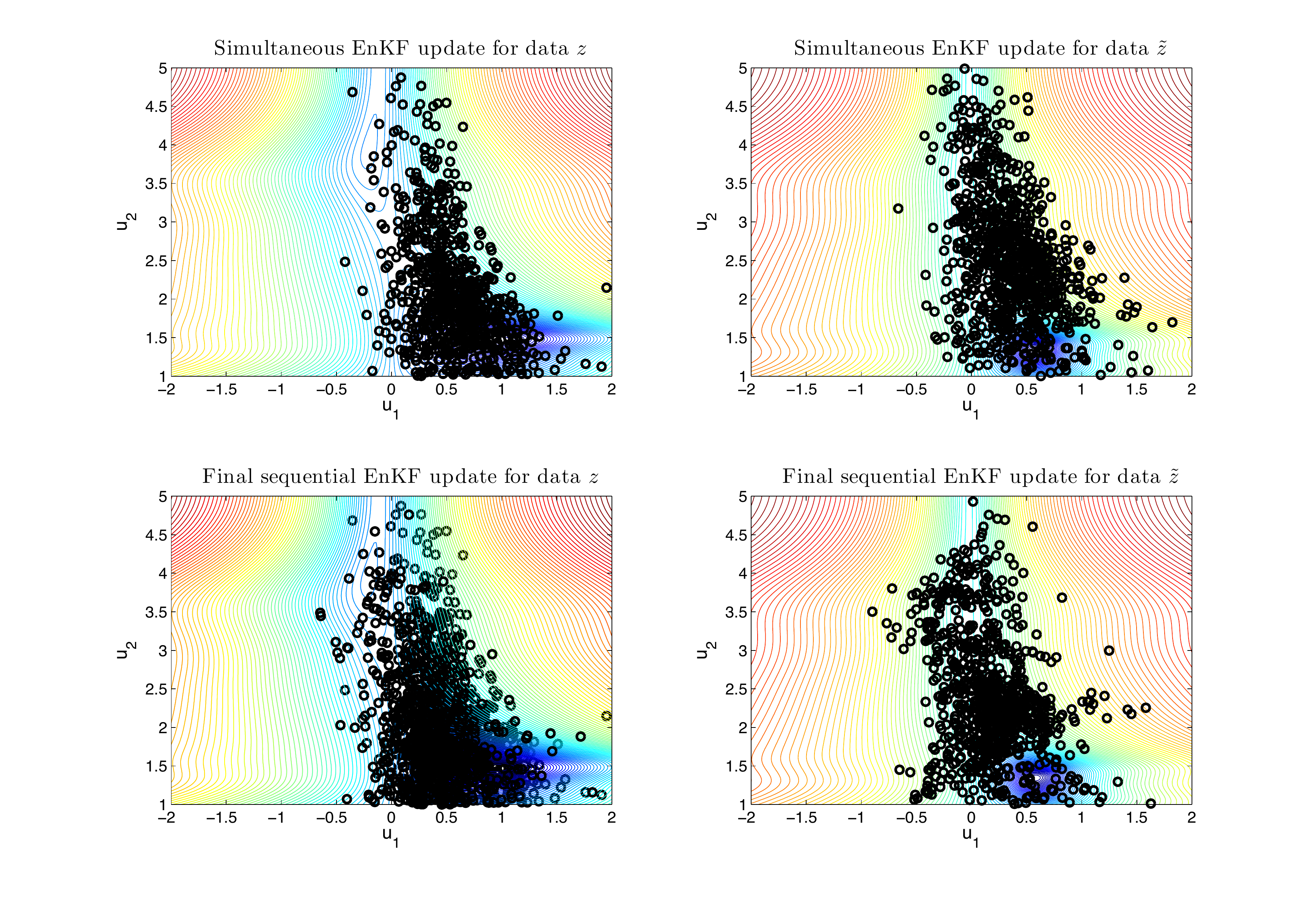}
\caption{Contours of the logarithm of the negative log posterior density and locations of $1,000$ members of the analysis EnKF-ensembles resulting from simultaneous and sequential updating for the two different sets $z$ and $\tilde z$ of observation data.}
\label{fig:RLC_post}
\end{figure}
We observe again that, for different data sets, the EnKF results in an ensemble which follows a distribution
which is, in one case, quite close and, in the other, quite far away from the true posterior distribution.
Interestingly, the difference between the two updating schemes does not seem to be too large.
This also holds  true for the means of the EnKF analysis ensembles when compared to the true posterior means in Table \ref{tab:RLC_1} for both data sets $z$ and $\tilde z$.

\begin{table}
\centering
\begin{tabular}{c|c|c||c|c}
update & EnKF mean & posterior mean & EnKF mean & posterior mean\\
& for data $z$ & for data $z$ & for data $\tilde z$ & for data $\tilde z$\\ \hline
1			& (0.42, 1.56)	& (0.42, 2.42)	& (0.27, 2.25)	& (0.35, 2.61)\\
2			& (0.44, 1.53)	& (0.39, 2.36)	& (0.20, 2.20)	& (0.32, 2.56)\\
3			& (0.43, 1.59)	& (0.38, 2.34)	& (0.19, 2.26)	& (0.31, 2.52)\\
4			& (0.43, 1.59)	& (0.38, 2.32)	& (0.19, 2.24)	& (0.30, 2.50)\\
\text{Simu.}	& (0.58, 1.84)	& (0.38, 2.32)	& (0.38, 2.40)	& (0.30, 2.50)
\end{tabular}
\caption{Means of the EnKF analysis ensembles and corresponding true posterior means for data $z$ (left) and $\tilde z$ (right).}\label{tab:RLC_1}
\end{table}

%
%

Finally, we are again interested in the marginals of the posterior and the associated histograms of the EnKF analysis ensembles which give us a rough impression of the difference between the distribution of the analysis variable and the true posterior.
In Fig. \ref{fig:RLC_Hist_AAO} we compare both marginals for the $4$th update and the simultaneous analysis ensemble for both data sets.
The distribution of the simultaneous EnKF analysis ensemble should not depend on the data whereas the distribution of the final EnKF analysis ensemble for the sequential updating clearly does in this example.
This is certainly caused by the nonlinearity of the forward map $G$: in the sequential updating the former analysis variable $U^a_n$ serves as initial one for the current update step $n+1$, therefore, the difference in the mean of the former analysis variables $U^a_n, \tilde U^a_n$ for different data sets $z, \tilde z$ might yield different forecast RVs $G(U^a_n), G(\tilde U^a_n)$ due to the nonlinearity of $G$ which yields different next analysis variables $U^a_{n+1}, \tilde U^a_{n+1}$.

\begin{figure}[h]
\hfill
\begin{minipage}{0.49\textwidth}
	\includegraphics[width = \textwidth]{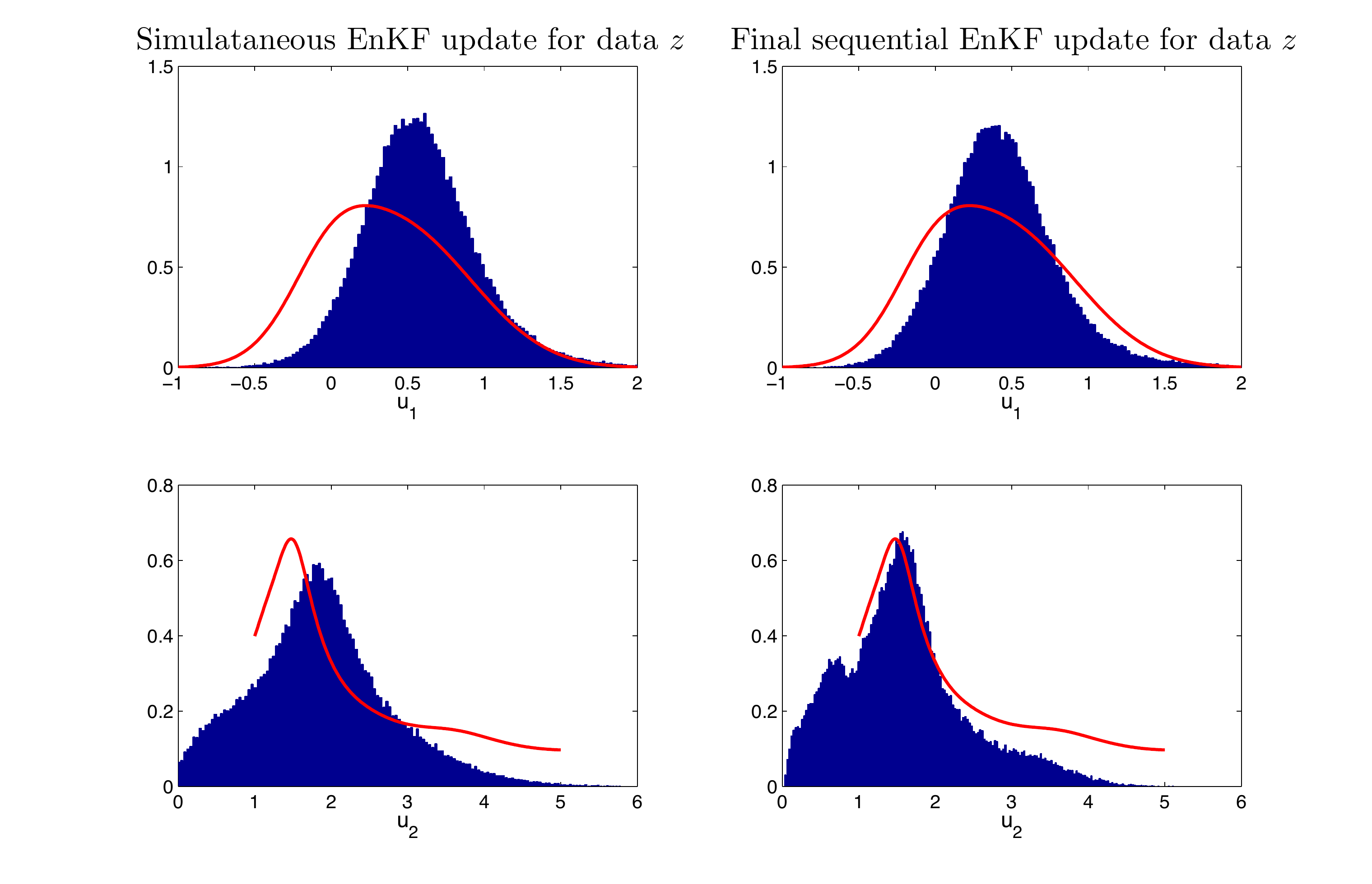}
\end{minipage}
\hfill
\begin{minipage}{0.49\textwidth}
	\includegraphics[width = \textwidth]{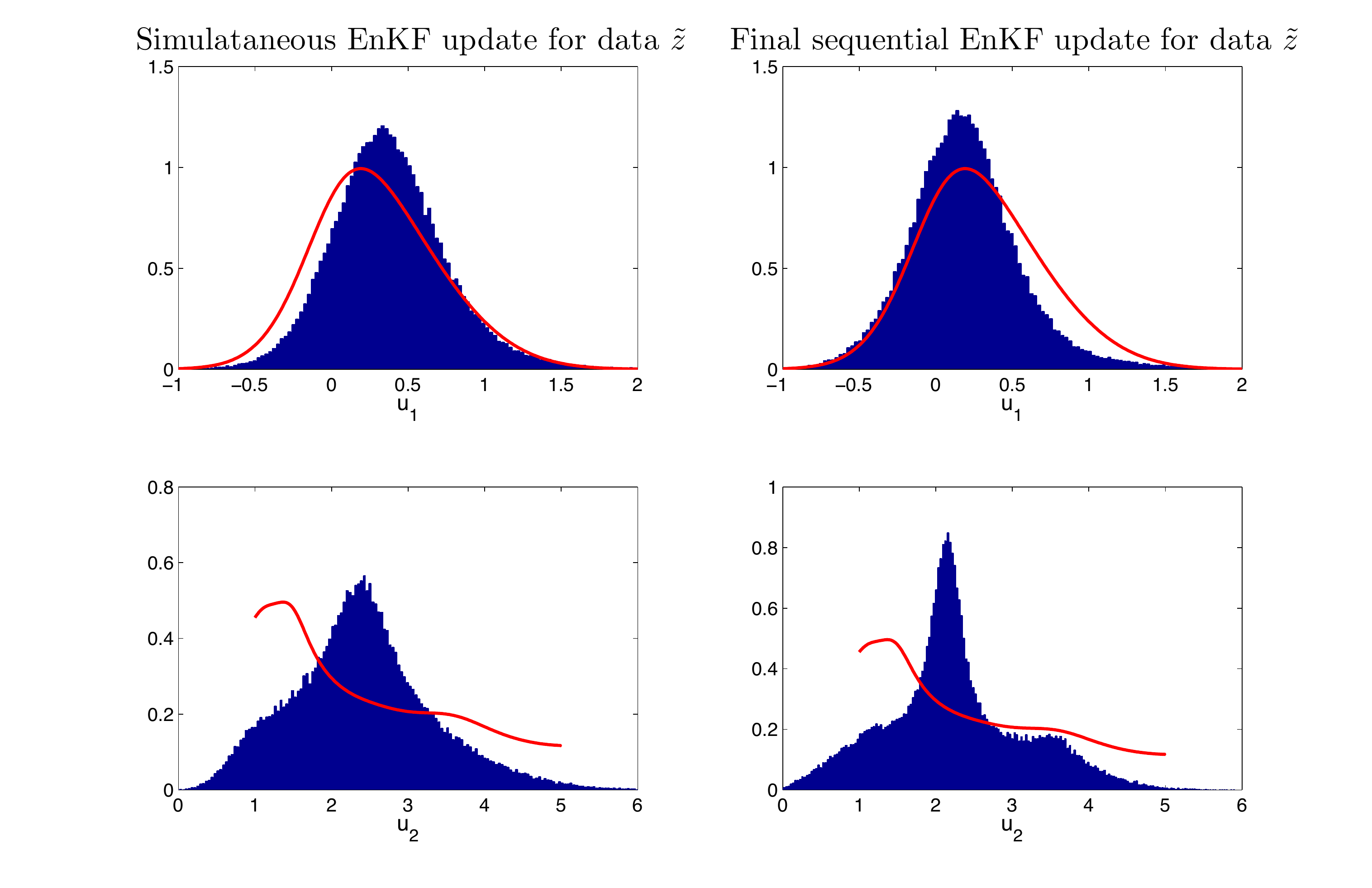}
\end{minipage}
\hfill
\caption{Posterior marginals and relative frequencies of the final EnKF analysis ensembles in $u_1$, $u_2$ for $z$ (left part) and $\tilde z$ (right part).}
\label{fig:RLC_Hist_AAO}
\end{figure}

%
%
\section{Conclusions} \label{sec:Conclusions}

We have given a detailed analysis of two popular generalized Kalman filtering methods, the EnKF and PCKF, applied to nonlinear (stationary) Bayesian inverse problems.
We recalled the Bayesian approach to inverse problems and its solution, the posterior measure, in a Hilbert space setting, for which we slightly generalized existing results concerning the well-posedness of Bayesian inverse problems.
Further, in order to characterize Kalman filter methods in the Bayesian framework, we also described Bayes estimators and highlighted    the distinction between the two objectives of inference and identification in Bayesian inversion realized by the posterior measure and Bayes estimators, respectively.
We then proved the convergence of the approximations provided by the EnKF and PCKF to a so-called analysis random variable in the large ensemble and large polynomial basis limit, respectively, reaffirming the fact that both methods are merely different numerical discretizations of the same updating scheme for random variables.
Moreover, the relation of both Kalman filter methods to a specific Bayes estimator, the linear posterior mean estimator, followed from this.
Hence, this work shows that the EnKF and PCKF are methods suited for identification -- providing in addition the random a priori estimation error --  rather than methods for rigorous inference in the sense of (regular versions of) the conditional measure.
Several carefully chosen numerical examples were given to illustrate these basic differences.

\bibliographystyle{siam}
\bibliography{literature_ernst_sprungk_starkloff}

\end{document}